\documentclass[11pt]{article}
\usepackage[margin=1in]{geometry}
\usepackage{amsmath,amssymb,amsthm, amsfonts}
\usepackage{enumitem}
\usepackage{fancyhdr}
\usepackage{tikz}
\usepackage[colorlinks=true,linkcolor=blue,citecolor=blue,urlcolor=blue]{hyperref}
\newcommand\shorttitle{No nontrivial uniquely $C_9$-saturated graphs}
\newcommand\authors{Fayzan Khan}
\newtheorem{conjecture}{Conjecture}

\newtheorem{theorem}{Theorem}[section]
\newtheorem{lemma}[theorem]{Lemma}
\newtheorem{proposition}[theorem]{Proposition}
\newtheorem{corollary}[theorem]{Corollary}
\theoremstyle{definition}

\newtheorem{remark}[theorem]{Remark}

\newcommand{\Forb}{\mathrm{Forb}}
\newcommand{\dist}{\mathrm{dist}}

\title{\textbf{There are no nontrivial uniquely $C_9$-saturated graphs}}
\author{Fayzan Khan}
\date{\today}

\begin{document}

\maketitle
\begin{abstract}
A graph $G$ is \emph{uniquely $C_t$-saturated} if $G$ contains no cycle of
length $t$ and, for every edge $e$ of the complement, $G+e$ contains exactly
one cycle of length $t$; it is \emph{nontrivial} if it has at least $t$
vertices. Wenger and West proved that no nontrivial uniquely $C_6$- or
$C_7$-saturated graphs exist, and conjectured the same for every $t\ge 6$; the
case $t=8$ was verified by them but the proof was never published, and the
cases $t\ge 9$ have remained open. We prove the conjecture for $t=9$: there is
no nontrivial uniquely $C_9$-saturated graph. The proof is a case analysis on
the length $L\in\{4,6,8,10,12\}$ of a longest even cycle of length at most
$12$, which such a graph must contain. The cases $L=12$, $L=10$, and $L=4$ are
settled entirely by hand; for the cases $L=6$ and $L=8$, hand
classifications of the components outside the cycle reduce each case to an
explicitly bounded finite family of configurations, which is then eliminated
by a short, fully replayable computer enumeration with known-answer controls
and with completeness certified by zero cap hits.
Along the way we prove, in sharpened form, the $t=8$ instance of a structural
lemma that Wenger and West stated without proof.

\bigskip


\end{abstract}

\section{Introduction}\label{sec:intro}

For a graph $F$, a graph $G$ is \emph{$F$-saturated} if $F$ is not a subgraph
of $G$ but $F$ is a subgraph of $G+e$ for every edge $e$ of the complement
$\overline G$. Saturation, introduced by Erd\H{o}s, Hajnal and Moon
\cite{EHM}, is classically studied through the minimum number of edges of an
$F$-saturated graph; see the dynamic survey \cite{FFS} for an overview.

Cooper, Lenz, LeSaulnier, Wenger and West \cite{CLLWW} initiated the study of
a much more rigid variant: $G$ is \emph{uniquely $F$-saturated} if adding any
edge of the complement to $G$ completes \emph{exactly one} copy of $F$. When
$F$ has $t$ vertices, every complete graph on fewer than $t$ vertices is
vacuously uniquely $F$-saturated; a uniquely $C_t$-saturated graph is
therefore called \emph{nontrivial} if it has at least $t$ vertices, where
$C_t$ denotes the cycle on $t$ vertices.

For short cycles, nontrivial examples exist and have been classified: the
uniquely $C_3$-saturated graphs are exactly the stars and the Moore graphs of
diameter two, there are exactly ten uniquely $C_4$-saturated graphs
\cite{CLLWW}, and the nontrivial uniquely $C_5$-saturated graphs are exactly
the friendship graphs \cite[Theorem 3.1]{WW}. In sharp contrast, Wenger and
West \cite{WW} proved that there are \emph{no} nontrivial uniquely
$C_6$-saturated graphs \cite[Theorem~4.1]{WW} and no nontrivial uniquely
$C_7$-saturated graphs \cite[Theorem~4.2]{WW}, and made the following
conjecture.

\begin{conjecture}[{Wenger--West \cite[Conjecture~1.2]{WW}}]\label{conj:ww}
For $t\ge 6$ there are no nontrivial uniquely $C_t$-saturated graphs.
\end{conjecture}

They also proved \cite[Theorem~1.1]{WW} that for every $t\ge 6$ there are only
finitely many uniquely $C_t$-saturated graphs; the proof is a nested induction
and yields no explicit bound. Concerning $t=8$ they write, immediately after
the conjecture: ``We have verified Conjecture~1.2 for $t=8$, but the proof is
quite long and does not contain any new ideas beyond those used in the proofs
of Theorems~4.1 and~4.2; thus we do not include the proof here.'' To the best
of our knowledge no proof for $t=8$ has appeared anywhere in the intervening
decade, and no case $t\ge 9$ has been settled in either direction. The
smallest case of Conjecture~\ref{conj:ww} that is open in print is thus $t=8$,
and the smallest case open even informally is $t=9$.

In this paper we settle the case $t=9$.

\begin{theorem}\label{thm:main}
There is no nontrivial uniquely $C_9$-saturated graph. That is,
Conjecture~\ref{conj:ww} holds for $t=9$.
\end{theorem}

\subsection*{Method}

Our starting point is the structural toolkit of \cite{WW}: a nontrivial
uniquely $C_9$-saturated graph $G$ may be assumed $2$-connected, contains no
cycle of length $14$ or $16$, contains no $2k$-cycle with a pendant path of
length $8-k$, has no pair of nonadjacent vertices with the same neighbourhood,
and---the key entry point---contains an even cycle of length at most
$2t-6=12$ \cite[Lemma~2.10]{WW}. We fix a \emph{longest} even cycle $C$ of
length at most $12$, so that its length $L$ lies in $\{4,6,8,10,12\}$ and $G$
contains no even cycle of length strictly between $L$ and $14$. The proof of
Theorem~\ref{thm:main} is a complete analysis of the five cases.

The cases $L=12$, $L=10$ and $L=4$ are closed by hand
(Sections~\ref{sec:l12}--\ref{sec:l4}): pendant-path and cycle-length
bookkeeping forces an explicit short list of candidate graphs, and in each of
them some nonadjacent pair fails to have exactly one path on nine vertices.

The cases $L=6$ and $L=8$ are genuinely harder; they are where, in Wenger and
West's words about growing $t$, ``more cases and details are needed to exclude
the shorter even cycles.'' For $L=6$ (Section~\ref{sec:l6}) we develop a chain of structural
lemmas---a split lemma at common neighbours, a \emph{corner lemma} for
common neighbours with small neighbourhoods, a domination lemma, and a swap
lemma for involutive automorphisms---which prove that every vertex outside
$C$ has a neighbour on $C$, and that every component of $G[R]$ is a
singleton, an edge, or a component confined to a single antipodal pair of
$C$ joined by a chord; the residual configurations form an explicit finite
family which a two-phase enumeration eliminates.
For $L=8$ (Section~\ref{sec:l8}) the central new observations are a
\emph{distance-$2$ attachment theorem} (no vertex outside $C$ has two
neighbours on $C$ at $C$-distance $2$; Lemma~\ref{lem:no-dist2}) and a
\emph{Hamiltonian-path kill lemma} (Lemma~\ref{lem:hamkill}) relating the
chords of $C$ to the possible attachment sets of outside vertices through the
Hamiltonian paths of the graph induced on $V(C)$. These close the chordless
sub-case by hand and cut the $2^{20}$ possible chord sets down to $6{,}891$,
of which only $300$ admit any outside structure at all.

The two residual families are eliminated by two small computer enumerations
(Appendix~\ref{app:enum}), each with its scope declared in advance, each
completing without truncation, and each guarded by known-answer controls run
in the same process before the main computation. The enumerations are the only
two non-hand steps in the paper; everything they use is stated as a lemma and
proved by hand, so that a reader can independently re-implement either
enumeration from the text alone. The two scripts are provided as ancillary
files.

\subsection*{A lemma of Wenger and West, sharpened at $t=9$}

Lemma~2.11 of \cite{WW} asserts, for $t\ge 7$ and without proof, structural
control over the vertices outside a longest even cycle of length $2t-6$. Our
analysis of the case $L=12$ proves the $t=9$ instance in sharpened form: the
vertices outside the cycle form an independent set, each has exactly two
neighbours on the cycle, and the two neighbours are at $C$-distance exactly
$3$---or $1$, a possibility the general statement excludes but which we
eliminate directly (Remark~\ref{rem:ww211}).

\subsection*{Organisation}

Section~\ref{sec:prelim} collects definitions, the facts imported from
\cite{WW}, and the general-purpose lemmas. Sections~\ref{sec:l12},
\ref{sec:l10} and \ref{sec:l4} close the cases $L=12$, $10$ and $4$.
Section~\ref{sec:l6} closes $L=6$ and Section~\ref{sec:l8} closes $L=8$.
Section~\ref{sec:conclusion} assembles the main theorem and discusses the
status of $t=8$ and $t=10$. Appendix~\ref{app:enum} documents the two
enumerations and their controls.

\section{Preliminaries}\label{sec:prelim}

All graphs are finite and simple. A \emph{$k$-path} is a path on $k$ vertices
(hence with $k-1$ edges); a \emph{$k$-cycle} is a cycle on $k$ vertices. The
\emph{length} of a path or cycle is its number of edges. For $k\ge 2$ and
$\ell\ge 0$ let $H_{k,\ell}$ denote the graph consisting of a $2k$-cycle
together with a pendant path of length $\ell$ attached at one of its vertices.
Since $H_{k,\ell}\subseteq H_{k,\ell'}$ for $\ell\le\ell'$, forbidding
$H_{k,\ell}$ as a subgraph forbids every $2k$-cycle with a pendant path of
length at least $\ell$.

Throughout, $G$ denotes a putative nontrivial uniquely $C_9$-saturated graph;
thus $|V(G)|\ge 9$, $G$ has no $9$-cycle, and $G+e$ has exactly one $9$-cycle
for every edge $e$ of $\overline G$. We will use the following facts from
\cite{WW}, specialised to $t=9$.

\begin{enumerate}[label=(F\arabic*),leftmargin=3.2em]
\item\label{f:2conn} \emph{($2$-connectivity.)} We may assume $G$ is
$2$-connected \cite[Corollary~2.5]{WW}.
\item\label{f:c14c16} $G$ contains no $C_{14}$ and no $C_{16}$
\cite[Lemma~2.7]{WW}.
\item\label{f:evencycle} $G$ contains an even cycle of length at most $12$
\cite[Lemma~2.10]{WW}.
\item\label{f:pendant} For $2\le k\le 7$, $G$ contains no $H_{k,8-k}$
\cite[Lemma~2.6]{WW}. Concretely, $G$ contains no $H_{2,6}$, $H_{3,5}$,
$H_{4,4}$, $H_{5,3}$, $H_{6,2}$, or $H_{7,1}$; by monotonicity, no $2k$-cycle
with a pendant path of length $\ge 8-k$. Figure~\ref{fig:forbidden_c9}
depicts all six of these forbidden subgraphs.

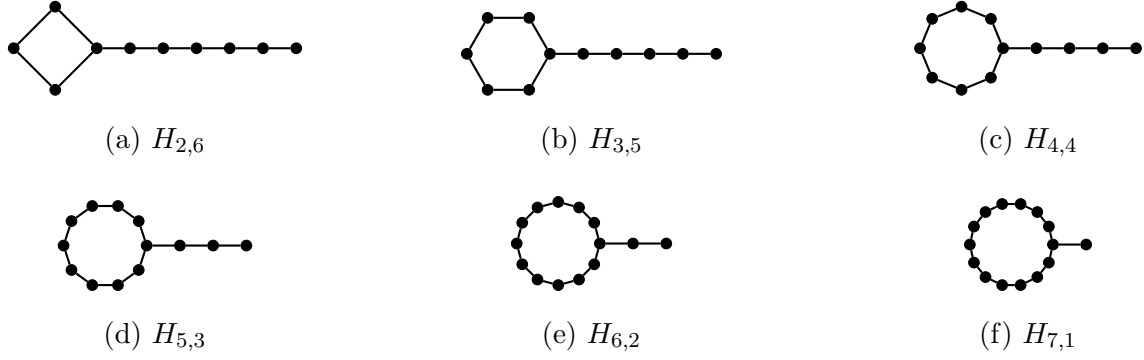
\begin{figure}[htbp]
  \centering
  \tikzset{
    vertex/.style={circle, draw, fill=black, inner sep=0pt, minimum size=4pt},
    edge/.style={thick}
  }


  \begin{minipage}[b]{0.30\textwidth}
    \centering
    \begin{tikzpicture}[scale=0.55]
      \foreach \i in {1,...,4} {
        \node[vertex] (v\i) at ({(\i-1)*90}:1.0) {};
      }
      \draw[edge] (v1) -- (v2) -- (v3) -- (v4) -- (v1);

      \node[vertex] (p1) at (1.8, 0) {};
      \node[vertex] (p2) at (2.6, 0) {};
      \node[vertex] (p3) at (3.4, 0) {};
      \node[vertex] (p4) at (4.2, 0) {};
      \node[vertex] (p5) at (5.0, 0) {};
      \node[vertex] (p6) at (5.8, 0) {};

      \draw[edge] (v1) -- (p1) -- (p2) -- (p3) -- (p4) -- (p5) -- (p6);
    \end{tikzpicture}
    \par\vspace{6pt}
    (a) $H_{2,6}$
  \end{minipage}\hfill
  \begin{minipage}[b]{0.30\textwidth}
    \centering
    \begin{tikzpicture}[scale=0.55]
      \foreach \i in {1,...,6} {
        \node[vertex] (v\i) at ({(\i-1)*60}:1.0) {};
      }
      \draw[edge] (v1) -- (v2) -- (v3) -- (v4) -- (v5) -- (v6) -- (v1);

      \node[vertex] (p1) at (1.8, 0) {};
      \node[vertex] (p2) at (2.6, 0) {};
      \node[vertex] (p3) at (3.4, 0) {};
      \node[vertex] (p4) at (4.2, 0) {};
      \node[vertex] (p5) at (5.0, 0) {};

      \draw[edge] (v1) -- (p1) -- (p2) -- (p3) -- (p4) -- (p5);
    \end{tikzpicture}
    \par\vspace{6pt}
    (b) $H_{3,5}$
  \end{minipage}\hfill
  \begin{minipage}[b]{0.30\textwidth}
    \centering
    \begin{tikzpicture}[scale=0.55]
      \foreach \i in {1,...,8} {
        \node[vertex] (u\i) at ({(\i-1)*45}:1.0) {};
      }
      \draw[edge] (u1) -- (u2) -- (u3) -- (u4) -- (u5) -- (u6) -- (u7) -- (u8) -- (u1);

      \node[vertex] (q1) at (1.8, 0) {};
      \node[vertex] (q2) at (2.6, 0) {};
      \node[vertex] (q3) at (3.4, 0) {};
      \node[vertex] (q4) at (4.2, 0) {};

      \draw[edge] (u1) -- (q1) -- (q2) -- (q3) -- (q4);
    \end{tikzpicture}
    \par\vspace{6pt}
    (c) $H_{4,4}$
  \end{minipage}

  \vspace{14pt}


  \begin{minipage}[b]{0.30\textwidth}
    \centering
    \begin{tikzpicture}[scale=0.55]
      \foreach \i in {1,...,10} {
        \node[vertex] (w\i) at ({(\i-1)*36}:1.0) {};
      }
      \draw[edge] (w1) -- (w2) -- (w3) -- (w4) -- (w5) -- (w6) -- (w7) -- (w8) -- (w9) -- (w10) -- (w1);

      \node[vertex] (r1) at (1.8, 0) {};
      \node[vertex] (r2) at (2.6, 0) {};
      \node[vertex] (r3) at (3.4, 0) {};

      \draw[edge] (w1) -- (r1) -- (r2) -- (r3);
    \end{tikzpicture}
    \par\vspace{6pt}
    (d) $H_{5,3}$
  \end{minipage}\hfill
  \begin{minipage}[b]{0.30\textwidth}
    \centering
    \begin{tikzpicture}[scale=0.55]
      \foreach \i in {1,...,12} {
        \node[vertex] (s\i) at ({(\i-1)*30}:1.0) {};
      }
      \draw[edge] (s1) -- (s2) -- (s3) -- (s4) -- (s5) -- (s6) -- (s7) -- (s8) -- (s9) -- (s10) -- (s11) -- (s12) -- (s1);

      \node[vertex] (t1) at (1.8, 0) {};
      \node[vertex] (t2) at (2.6, 0) {};

      \draw[edge] (s1) -- (t1) -- (t2);
    \end{tikzpicture}
    \par\vspace{6pt}
    (e) $H_{6,2}$
  \end{minipage}\hfill
  \begin{minipage}[b]{0.30\textwidth}
    \centering
    \begin{tikzpicture}[scale=0.55]
      \foreach \i in {1,...,14} {
        \node[vertex] (z\i) at ({(\i-1)*360/14}:1.0) {};
      }
      \draw[edge] (z1) -- (z2) -- (z3) -- (z4) -- (z5) -- (z6) -- (z7) -- (z8) -- (z9) -- (z10) -- (z11) -- (z12) -- (z13) -- (z14) -- (z1);

      \node[vertex] (y1) at (1.8, 0) {};

      \draw[edge] (z1) -- (y1);
    \end{tikzpicture}
    \par\vspace{6pt}
    (f) $H_{7,1}$
  \end{minipage}

  \caption{The six forbidden subgraphs $H_{k,8-k}$ for $2 \le k \le 7$ used in (F4): a $2k$-cycle with a pendant path of length $8-k$.}
  \label{fig:forbidden_c9}
\end{figure}

\item\label{f:twins} $G$ has no \emph{twins}: no two nonadjacent vertices with
equal neighbourhoods \cite[Lemma~2.9]{WW}.
\item\label{f:obs} \emph{(Unique-saturation reformulation
\cite[Observation~2.1]{WW}.)} $G$ is uniquely $C_9$-saturated if and only if
every two nonadjacent vertices of $G$ are the endpoints of exactly one
$9$-path, and no two adjacent vertices are the endpoints of a $9$-path
(equivalently, $G$ has no $9$-cycle).
\end{enumerate}

For nonadjacent vertices $u,v$ we write $P(u,v)$ for the unique $9$-path with
endpoints $u$ and $v$ guaranteed by \ref{f:obs}. Exhibiting a second
$9$-path between a nonadjacent pair, a $9$-cycle, a $C_{14}$ or $C_{16}$, or a
forbidden $H_{k,8-k}$, is how every configuration in this paper is
eliminated; we say the configuration is \emph{killed}. Killing is monotone:
all our killers are subgraphs together with a non-adjacency, and adding
further vertices or edges \emph{not} joining the named nonadjacent pair
preserves them.

\subsection*{The longest even cycle and its surroundings}

By \ref{f:evencycle} fix a longest even cycle
$C=[c_0,c_1,\dots,c_{L-1}]$ of length $L\le 12$, indices modulo $L$, so
$L\in\{4,6,8,10,12\}$, and set $R=V(G)\setminus V(C)$. By maximality $G$ has
no even cycle of length $L'$ with $L<L'\le 12$; combining this with
\ref{f:c14c16} and $C_9$-freeness, the set of forbidden cycle lengths is
\[
\Forb(L)\;=\;\{9,14,16\}\;\cup\;\{\text{even } L' : L<L'\le 12\}.
\]

An \emph{$m$-arc} of $C$ is a subpath of $C$ with $m$ edges. The
\emph{$C$-distance} of $c_i,c_j$ is $\min(|i-j|,\,L-|i-j|)$. A \emph{chord}
$c_ic_j$ has $C$-distance $d=\dist(c_i,c_j)\ge 2$. A vertex $x\in R$ with two
$C$-neighbours $c_i,c_j$ is \emph{attached at distance} $d=\dist(c_i,c_j)$.
A \emph{chordal path} is a path of length $\ge 2$ with both endpoints on $C$
and all interior vertices in $R$. Since $G$ is $2$-connected, every vertex of
$R$ lies on a chordal path (fan lemma). We write $N_C(x)=N(x)\cap V(C)$.

Two elementary counting facts are used throughout:
\begin{enumerate}[label=(E\arabic*),leftmargin=3.2em]
\item\label{e:chord} a chord at distance $d$ creates cycles of lengths $d+1$
and $L-d+1$;
\item\label{e:attach} a vertex attached at distance $d$ creates cycles of
lengths $d+2$ and $L-d+2$; more generally a chordal path of length $\ell$
between $C$-vertices at distance $s$ creates cycles of lengths $\ell+s$ and
$\ell+L-s$.
\end{enumerate}
Moreover, the cycle through a chord (or attachment, or chordal path) that uses
the $d$-arc of $C$ has the complementary $(L-d)$-arc as a pendant path of
length $L-d$ hanging at one of its vertices, and vice versa; combined with
\ref{f:pendant} this \emph{pendant-path device} is our main structural tool.

\subsection*{Two general lemmas}

The following two lemmas are used in the cases $L\le 8$; the first needs only
that $10,12,14,16\in\Forb(L)$, the second is completely general.

\begin{lemma}[Even-path lemma]\label{lem:evenpath}
Suppose $L\le 8$. Let $u,v$ be nonadjacent. Then every $u$--$v$ path of even
length $\ell\le 8$ shares an interior vertex with $P(u,v)$. In particular
every common neighbour of $u$ and $v$ is an interior vertex of $P(u,v)$.
\end{lemma}

\begin{proof}
Otherwise the path and $P(u,v)$ (of length $8$) are internally disjoint and
together form a cycle of length $8+\ell\in\{10,12,14,16\}\subseteq\Forb(L)$.
\end{proof}

\begin{lemma}[Swap lemma]\label{lem:swap}
Let $\sigma$ be an automorphism of $G$ and let $u,v$ be nonadjacent vertices
fixed by $\sigma$. Then $\sigma$ fixes every vertex of $P(u,v)$.
Consequently, if $w_1w_2$ is a component of $G[R]$ isomorphic to $K_2$ with
$N_C(w_1)=N_C(w_2)$ and $\deg(w_i)=|N_C(w_i)|+1$, then no $9$-path between
nonadjacent vertices outside $\{w_1,w_2\}$ passes through $w_1$ or $w_2$.
\end{lemma}

\begin{proof}
$\sigma(P(u,v))$ is a $9$-path with endpoints $u,v$, hence equals $P(u,v)$ by
uniqueness; a path mapped to itself with its endpoints fixed is fixed
pointwise. For the consequence, the transposition of $w_1$ and $w_2$ is an
automorphism of $G$; if some $P(u,v)$ with $u,v\notin\{w_1,w_2\}$ passed
through $w_1$, its image would be a second $9$-path from $u$ to $v$.
\end{proof}

We call a $K_2$-component as in Lemma~\ref{lem:swap} \emph{swappable}.

\section{Case \texorpdfstring{$L=12$}{L=12}}\label{sec:l12}

Here $\Forb(12)=\{9,14,16\}$, and by \ref{f:pendant} the relevant forbidden
subgraphs are $H_{2,6}$, $H_{3,5}$, $H_{4,4}$ and $H_{6,2}$.

\begin{lemma}\label{lem:12-indep}
$R$ is an independent set.
\end{lemma}

\begin{proof}
Suppose $uv$ is an edge with $u,v\in R$. Since $G-u$ is connected, a shortest
path from $v$ to $C$ in $G-u$, followed by the edge $vu$, is a pendant path
of length $\ge 2$ hanging at a vertex of the $12$-cycle $C$: an $H_{6,2}$.
\end{proof}

\begin{lemma}\label{lem:12-attach}
Every $x\in R$ has exactly two neighbours, both on $C$, at $C$-distance $1$
or $3$.
\end{lemma}

\begin{proof}
By Lemma~\ref{lem:12-indep} and $2$-connectivity, $x$ has at least two
neighbours, all on $C$. Two neighbours at distance $d$ give cycles of lengths
$d+2$ and $14-d$ by \ref{e:attach}, and a pendant path by the pendant-path
device. If $d=2$: the $4$-cycle $x c_0c_1c_2$ has the pendant path
$c_2c_3\cdots c_{11}$ of length $9\ge 6$: $H_{2,6}$. If $d=4$: a $6$-cycle
with a pendant path of length $7\ge 5$: $H_{3,5}$. If $d=6$: an $8$-cycle
with a pendant path of length $5\ge 4$: $H_{4,4}$. If $d=5$: $14-5=9$, a
$9$-cycle. So $d\in\{1,3\}$. Three neighbours would need pairwise distances
in $\{1,3\}$; on a $12$-cycle no triple of positions achieves this: starting
from a pair at distance $1$, say positions $\{0,1\}$, a third position must
lie in $\{1,11,3,9\}$ (distance $1$ or $3$ from $0$) and in $\{0,2,4,10\}$
(from $1$), which are disjoint; starting from $\{0,3\}$, the sets
$\{1,11,3,9\}$ and $\{2,4,6,0\}$ are disjoint. Hence exactly two neighbours.
\end{proof}

\begin{lemma}\label{lem:12-mult}
At most one vertex of $R$ is attached at distance $1$, and the vertices
attached at distance $3$ occupy pairwise distinct position pairs.
\end{lemma}

\begin{proof}
If $x$ is attached to $(c_0,c_1)$ and $y$ to $(c_a,c_{a+1})$ with $a\ne 0$,
then $x\,c_1\cdots c_a\,y\,c_{a+1}\cdots c_0\,x$ is a cycle of length
$2+(a-1)+2+(11-a)=14$ (for $a=1$, the cycle $x\,c_1\,y\,c_2\cdots c_0\,x$
also has length $14$), contradicting \ref{f:c14c16}. Two distance-$3$
vertices on the same pair are twins, contradicting \ref{f:twins}.
\end{proof}

We write $A\subseteq\mathbb Z_{12}$ for the set of positions $a$ such that
some $x_a\in R$ is attached to $(c_a,c_{a+3})$, and $z$ for the distance-$1$
vertex (if any), attached to $(c_b,c_{b+1})$.

\begin{lemma}\label{lem:12-chords}
Every chord of $C$ has distance $2$ or $6$.
\end{lemma}

\begin{proof}
By \ref{e:chord}: $d=3$ gives a $4$-cycle whose complementary arc is a
pendant path of length $8\ge 6$ ($H_{2,6}$); $d=4$ gives $5+9$, a $9$-cycle;
$d=5$ gives a $6$-cycle with a pendant path of length $6\ge 5$ ($H_{3,5}$).
\end{proof}

\begin{lemma}\label{lem:12-nochordA}
If $A\neq\emptyset$ then $C$ has no chord.
\end{lemma}

\begin{proof}
Let $x=x_0$ be attached to $(c_0,c_3)$; the reflection $i\mapsto 3-i$ is a
symmetry of the configuration fixing $x$. Consider first distance-$2$ chords
$c_kc_{k+2}$.
\begin{itemize}[leftmargin=2em]
\item $c_0c_2$ or $c_1c_3$: the $4$-cycle $x\,c_0c_2c_3$ (resp.\
$x\,c_0c_1c_3$) has a pendant path of length $8$ along $C$: $H_{2,6}$.
\item $c_2c_4$: the $6$-cycle $x\,c_3c_4c_2c_1c_0$ has the pendant path
$c_4c_5\cdots c_{11}$ of length $7$: $H_{3,5}$. The chord $c_{11}c_1$ is its
mirror image.
\item $c_kc_{k+2}$ for $3\le k\le 7$: the vertices $c_1,c_9$ are at
$C$-distance $4$ and nonadjacent (Lemma~\ref{lem:12-chords}), and the $8$-arc
$c_1c_2\cdots c_9$ is a $9$-path between them; the path
$c_1\,c_0\,x\,c_3\,c_4\cdots c_9$ modified to use the chord $c_kc_{k+2}$ to
skip $c_{k+1}$ is a second $9$-path ($3+6-1=8$ edges). This contradicts
\ref{f:obs}. For $k=8,9,10$ apply the reflection $i\mapsto 3-i$ (images of
$k=7,6,5$).
\end{itemize}
Now the diametric chords ($d=6$):
\begin{itemize}[leftmargin=2em]
\item $c_0c_6$ (and its mirror $c_3c_9$): the $6$-cycle $x\,c_3c_4c_5c_6c_0$
has the pendant path $c_6c_7\cdots c_{11}$ of length $5$: $H_{3,5}$.
\item $c_4c_{10}$, $c_5c_{11}$: the $6$-cycle $x\,c_0c_{11}c_{10}c_4c_3$
(resp.\ $x\,c_0c_{11}c_5c_4c_3$) has a pendant path of length $5$
($c_4c_5\cdots c_9$, resp.\ $c_5\cdots c_{10}$): $H_{3,5}$.
\item $c_1c_7$: the $8$-cycle $x\,c_0c_1c_7c_6c_5c_4c_3$ has the pendant path
$c_7c_8c_9c_{10}c_{11}$ of length $4$: $H_{4,4}$. $c_2c_8$: the $8$-cycle
$x\,c_3c_2c_8c_9c_{10}c_{11}c_0$ has the pendant path $c_8c_7c_6c_5c_4$:
$H_{4,4}$. \qedhere
\end{itemize}
\end{proof}

\begin{lemma}\label{lem:12-nochordZ}
If $z$ exists then $C$ has no chord.
\end{lemma}

\begin{proof}
Let $z$ be attached to $(c_0,c_1)$; the reflection $i\mapsto 1-i$ fixes $z$.
\begin{itemize}[leftmargin=2em]
\item $c_0c_2$ or $c_{11}c_1$: the $4$-cycle $z\,c_0c_2c_1$ (resp.\
$z\,c_1c_{11}c_0$) has a pendant path of length $9$: $H_{2,6}$.
\item $c_kc_{k+2}$ for $1\le k\le 6$: $c_0,c_8$ are at distance $4$ and
nonadjacent, the $8$-arc $c_0c_1\cdots c_8$ is a $9$-path, and
$c_0\,z\,c_1\cdots c_8$ with the chord skipping $c_{k+1}$ is a second one.
For $k=7,\dots,10$ apply $i\mapsto 1-i$ (images of $k=6,\dots,3$).
\item $c_0c_6$ or $c_1c_7$: the $8$-cycle $z\,c_1c_2\cdots c_6c_0$ has the
pendant path $c_6c_7c_8c_9$ of length $4$: $H_{4,4}$ (for $c_1c_7$, mirror).
\item $c_kc_{k+6}$ for $2\le k\le 5$: the $8$-cycle
$z\,c_1\cdots c_k\,c_{k+6}\cdots c_{11}c_0$ has the pendant path
$c_kc_{k+1}\cdots c_{k+5}$ of length $5$: $H_{4,4}$. \qedhere
\end{itemize}
\end{proof}

\begin{lemma}[$9$-arc lemma]\label{lem:12-9arc}
Suppose $A\neq\emptyset$ (so $C$ has no chord). For $a\in A$ the pair
$(c_a,c_{a+3})$ is nonadjacent, and for every $a'\in A$ with
$a'-a\in\{3,4,\dots,9\}\pmod{12}$ the walk
\[
c_a\,c_{a-1}\cdots c_{a'+3}\;x_{a'}\;c_{a'}\,c_{a'-1}\cdots c_{a+3}
\]
is a $9$-path between $c_a$ and $c_{a+3}$ (it has
$(9-(a'-a))+2+((a'-a)-3)=8$ edges). Hence, by uniqueness, each $a\in A$ has
at most one $a'\in A$ with $a'-a\notin\{0,\pm1,\pm2\}$. Consequently
$|A|\le 4$ and $A$ lies in a window of five consecutive positions; up to
rotation and reflection,
\[
A\in\bigl\{\{0\},\{0,1\},\{0,2\},\{0,3\},\{0,4\},\{0,5\},\{0,6\},
\{0,1,2\},\{0,1,3\},\{0,2,4\},\{0,1,2,3\}\bigr\}.
\]
\end{lemma}

\begin{proof}
Nonadjacency of $(c_a,c_{a+3})$ holds since $C$ has no chord. Distinct
admissible $a'$ give distinct $9$-paths between the same nonadjacent pair, so
at most one exists. For the combinatorial consequence, one checks directly
that any other $3$- or $4$-subset of $\mathbb Z_{12}$ has an element with two
partners at cyclic distance $\ge 3$ (and $\{0,2,3\}$ is the reflection of
$\{0,1,3\}$).
\end{proof}

\begin{lemma}[$8$-arc lemma]\label{lem:12-8arc}
If $z$ (attached to $(c_b,c_{b+1})$) and $x_a$ both exist, then
$b-a\in\{0,1,2\}\pmod{12}$.
\end{lemma}

\begin{proof}
Otherwise the pairs $\{c_b,c_{b+1}\}$ and $\{c_a,c_{a+3}\}$ are disjoint, or
share exactly one endpoint with $b=a+3$ or $b+1=a$; in each of these cases
both pairs lie inside a common $8$-arc of $C$. (With $s=b-a\bmod 12$: for
$s\in\{4,\dots,7\}$ the joint span is $s+1\le 8$; for $s\in\{8,9,10\}$ it is
$15-s\le 8$ in the other direction; for $s\in\{3,11\}$ it is $5$.) The two
ends of that $8$-arc are at $C$-distance $4$ and nonadjacent (no chords, by
Lemmas~\ref{lem:12-nochordA} and~\ref{lem:12-nochordZ}), and the $8$-arc
itself is a $9$-path between them. Replacing the sub-arc $c_a\cdots c_{a+3}$
by $c_a\,x_a\,c_{a+3}$ (one edge fewer) and the edge $c_bc_{b+1}$ by
$c_b\,z\,c_{b+1}$ (one edge more) gives a second $9$-path between the same
pair.
\end{proof}

\begin{lemma}\label{lem:12-triples}
If $\{a,a+1,a+2\}\subseteq A$ then
$c_a x_a c_{a+3} c_{a+2} x_{a+2} c_{a+5} c_{a+4} x_{a+1} c_{a+1} c_a$ is a
$9$-cycle. If $\{a,a+1,a+3\}\subseteq A$ then
$c_a x_a c_{a+3} x_{a+3} c_{a+6} c_{a+5} c_{a+4} x_{a+1} c_{a+1} c_a$ is a
$9$-cycle. Hence $A\notin\{\{0,1,2\},\{0,1,3\},\{0,1,2,3\}\}$ (up to
symmetry).
\end{lemma}

\begin{proof}
Direct verification of the two cycles (nine vertices each, consecutive
vertices adjacent).
\end{proof}

\begin{proposition}\label{prop:l12}
No graph exists in Case $L=12$.
\end{proposition}

\begin{proof}
By Lemmas~\ref{lem:12-indep}--\ref{lem:12-triples},
$V(G)=V(C)\cup\{x_a:a\in A\}\cup Z$ where $Z$ is $\{z\}$ or empty, $A$ is one
of the sets allowed by Lemmas~\ref{lem:12-9arc} and~\ref{lem:12-triples}, and
$C$ has no chord unless $A\cup Z=\emptyset$. In each sub-case below we
exhibit a nonadjacent pair with no $9$-path (or two), contradicting
\ref{f:obs}. All paths between the named pair are enumerated in a graph on at
most $15$ vertices in which $R$ is independent, so every path alternates arcs
of $C$ with single-vertex ``detours'' through $R$; the enumerations are
exhaustive.

\emph{(a) $A=\{0\}$}, $z$ on $(c_b,c_{b+1})$ with $b\in\{0,1,2\}$
(Lemma~\ref{lem:12-8arc}) or absent. The pair $(c_0,c_3)$ is nonadjacent, and
$x_0$ cannot be interior to any $c_0$--$c_3$ path (both its neighbours are
the endpoints). The available paths run along the $3$-arc (length $3$; $4$ if
$z$ is used on it) or along the $9$-arc (length $9$; $10$ with $z$). No
$9$-path.

\emph{(b) $A=\{0,1\}$}, $z$ absent or on $(c_1,c_2)$ or $(c_2,c_3)$
(Lemma~\ref{lem:12-8arc} applied to $a=0$ and $a=1$). Paths from $c_0$ to
$c_3$ avoiding $x_0$: $c_0c_1c_2c_3$ (length $3$; with $z$: $4$);
$c_0c_1x_1c_4c_3$ (length $4$); the $9$-arc (length $9$; with $z$: $10$); the
$9$-arc rerouted through $x_1$, i.e.\ $c_0c_{11}\cdots c_4x_1c_1c_2c_3$
(length $12$; with $z$: $13$). No $9$-path.

\emph{(c) $A=\{0,2\}$}, $z$ absent or on $(c_2,c_3)$. Paths from $c_0$ to
$c_3$ avoiding $x_0$: $c_0c_1c_2c_3$ ($3$), $c_0c_1c_2zc_3$ ($4$),
$c_0c_1c_2x_2c_5c_4c_3$ ($6$), the $9$-arc ($9$),
$c_0c_{11}\cdots c_5x_2c_2c_3$ ($10$), $c_0c_{11}\cdots c_5x_2c_2zc_3$
($11$). No $9$-path.

\emph{(d) $A=\{0,m\}$, $3\le m\le 6$} (the cases $m=7,8,9$ are reflections of
$5,4,3$). Here $z$ is absent, since Lemma~\ref{lem:12-8arc} would require
$b\in\{0,1,2\}\cap\{m,m+1,m+2\}=\emptyset$. The vertices $x_0$ and $x_m$ are
nonadjacent, and every $x_0$--$x_m$ path has the form $x_0\,c_i\,(\text{arc})
\,c_j\,x_m$ with $i\in\{0,3\}$, $j\in\{m,m+3\}$, in one of two directions
around $C$; its length is one of
\[
m-1,\quad 17-m,\quad m+2,\quad 14-m,\quad m+2,\quad 14-m,\quad m+5,\quad 11-m.
\]
The number of these equal to $8$ is: two for $m=3$ (namely $11-m$ and $m+5$),
zero for $m=4$, zero for $m=5$, four for $m=6$. In every case the count is
not one.

\emph{(e) $A=\{0,2,4\}$.} Again $z$ is absent. Paths from $c_2$ to $c_5$
avoiding $x_2$: $c_2c_3c_4c_5$ ($3$), $c_2c_3c_4x_4c_7c_6c_5$ ($6$),
$c_2c_1c_0x_0c_3c_4c_5$ ($6$), $c_2c_1c_0x_0c_3c_4x_4c_7c_6c_5$ ($9$
edges---but this path has $10$ vertices, i.e.\ length $9$, not a $9$-path),
$c_2c_1c_0c_{11}\cdots c_5$ ($9$),
$c_2c_1c_0c_{11}\cdots c_7x_4c_4c_5$ ($10$),
$c_2c_3x_0c_0c_{11}\cdots c_5$ ($10$),
$c_2c_3x_0c_0c_{11}\cdots c_7x_4c_4c_5$ ($11$). None has length $8$: no
$9$-path.

\emph{(f) $A=\emptyset$, $z$ present} on $(c_0,c_1)$: by
Lemma~\ref{lem:12-nochordZ} there are no chords, so $G=C_{12}+z$, and the
nonadjacent pair $(c_0,c_2)$ has paths of lengths $2,3,10,11$ only.

\emph{(g) $A=\emptyset$, $z$ absent}, so $V(G)=V(C)$ and all structure is
chords, at distance $2$ or $6$ (Lemma~\ref{lem:12-chords}). First, a
diametric chord excludes every distance-$2$ chord: with $c_0c_6$ present, a
chord $c_ic_{i+2}$ with $\{i,i+1,i+2\}$ contained in $\{0,\dots,6\}$ or in
$\{6,\dots,12\}$ yields a $6$-cycle (e.g.\ $c_0\cdots c_ic_{i+2}\cdots
c_6c_0$) with a pendant path of length $5$ along the other half of $C$:
$H_{3,5}$; and the crossing chords $c_5c_7$, $c_{11}c_1$ yield the $8$-cycle
$c_1c_{11}c_{10}\cdots c_6c_0c_1$\,---\,using $c_0c_6$\,---\,with pendant
path $c_1c_2c_3c_4c_5$ of length $4$: $H_{4,4}$.

\emph{(g1) Only diametric chords}, forming a set $D\subseteq\mathbb Z_6$ of
diameters. If $|D|\ge 3$ then $D$ contains a triple of type $\{i,i+1,i+2\}$,
$\{i,i+1,i+3\}$ or $\{i,i+2,i+4\}$ (the only types up to symmetry), and each
contains a $9$-cycle:
$c_0c_6c_5c_4c_3c_2c_8c_7c_1c_0$;\;
$c_0c_6c_5c_4c_3c_9c_8c_7c_1c_0$;\;
$c_0c_6c_5c_4c_{10}c_9c_8c_2c_1c_0$ respectively. If $D=\{0,m\}$: for $m=1$
the pair $(c_3,c_5)$ has paths of lengths $2,5,5,10,10$ only; for $m=2$ the
pair $(c_4,c_{10})$ has the two $9$-paths
$c_4c_5c_6c_0c_1c_2c_8c_9c_{10}$ and $c_4c_3c_2c_8c_7c_6c_0c_{11}c_{10}$;
for $m=3$ the pair $(c_0,c_2)$ has paths of lengths $2,5,5,6,10$ only. If
$|D|=1$: the pair $(c_1,c_3)$ has paths of lengths $2,5,10$ only. If
$D=\emptyset$: the pair $(c_0,c_3)$ has paths of lengths $3,9$ only.

\emph{(g2) Only distance-$2$ chords.} Suppose two chords $c_0c_2$ and
$c_mc_{m+2}$ ($0<m\le 10$) are present. For $m=1$ (crossing):
$c_0c_2c_1c_3c_4\cdots c_8$ is a second $9$-path for the distance-$4$ pair
$(c_0,c_8)$. For $m=2$ (sharing an endpoint): the distance-$3$ pair
$(c_0,c_9)$\,---\,going the long way\,---\,has the two $9$-paths
$c_0c_2c_3\cdots c_9$ and $c_0c_1c_2c_4\cdots c_9$. For $3\le m\le 7$: both
chords lie on the $9$-arc $c_0\cdots c_9$, and each alone shortcuts it to a
$9$-path: two $9$-paths for $(c_0,c_9)$. For $m=8,9,10$: reflections
($i\mapsto 2-i$) of $m=4,3,2$. So at most one chord exists. With exactly one
chord $c_0c_2$: the pair $(c_{10},c_1)$ (distance $3$) has paths of lengths
$3,9,10$ only. With no chord: $(c_0,c_3)$ has lengths $3,9$ only.

Every sub-case contradicts \ref{f:obs}.
\end{proof}

\begin{remark}\label{rem:ww211}
Lemmas~\ref{lem:12-indep}--\ref{lem:12-mult} prove the $t=9$ instance of
\cite[Lemma~2.11]{WW} (stated there for $t\ge 7$ without proof) in sharpened
form: outside a longest even cycle of length $2t-6$ the vertices form an
independent set, each with exactly two neighbours on the cycle at odd
$C$-distance---and at $t=9$ the odd distance $5$ is impossible (it creates a
$9$-cycle), so only distance $3$, or the degenerate distance $1$ handled by
Lemmas~\ref{lem:12-mult} and~\ref{lem:12-8arc}, can occur. At $t=7$ the
analogous computation leaves only distance $1$, matching Case~1 of the proof
of \cite[Theorem~4.2]{WW}. We found no inconsistency with \cite{WW}.
\end{remark}

\section{Case \texorpdfstring{$L=10$}{L=10}}\label{sec:l10}

Here $\Forb(10)=\{9,12,14,16\}$ and the relevant forbidden subgraphs are
$H_{2,6}$, $H_{3,5}$, $H_{4,4}$, $H_{5,3}$.

\begin{lemma}\label{lem:10-chords}
Every chord of $C$ has distance $4$ or $5$.
\end{lemma}

\begin{proof}
$d=2$: cycles $3$ and $9$ by \ref{e:chord}, a $9$-cycle. $d=3$: a $4$-cycle
with a pendant path of length $7\ge 6$: $H_{2,6}$.
\end{proof}

\begin{lemma}\label{lem:10-chordal}
Every chordal path has length at most $3$; a chordal path of length $3$ joins
$C$-vertices at distance exactly $2$.
\end{lemma}

\begin{proof}
A chordal path of length $\ge 4$ contains a pendant path of length $3$
hanging at $C$: $H_{5,3}$. For length $3$ at end-distance $s$: $s=1$ gives a
$4$-cycle with pendant path of length $9$ ($H_{2,6}$); $s=3$ gives a
$6$-cycle with pendant path of length $7$ ($H_{3,5}$); $s=4$ gives
$3+10-4=9$, a $9$-cycle; $s=5$ gives an $8$-cycle with pendant path of
length $4$ ($H_{4,4}$).
\end{proof}

\begin{lemma}\label{lem:10-attach}
A vertex of $R$ with two $C$-neighbours has them at distance $1$ or $5$; a
vertex of $R$ all of whose neighbours lie on $C$ has exactly two of them.
\end{lemma}

\begin{proof}
$d=2$: a $4$-cycle with pendant path of length $8$ ($H_{2,6}$). $d=3$:
cycles $5$ and $9$, a $9$-cycle. $d=4$: a $6$-cycle with pendant path of
length $6$ ($H_{3,5}$). No triple of positions on a $10$-cycle has pairwise
distances in $\{1,5\}$.
\end{proof}

\begin{lemma}\label{lem:10-comp}
Every component of $G[R]$ is a single vertex (a \emph{singleton}, attached at
distance $1$ or $5$) or a single edge $uv$ whose two vertices have degree $2$
in $G$, forming a chordal path $c_i\,u\,v\,c_{i+2}$.
\end{lemma}

\begin{proof}
Let $uv$ be an edge of $G[R]$; it lies on a chordal path, which by
Lemma~\ref{lem:10-chordal} is $c_i\,u\,v\,c_j$ with $j=i\pm 2$. If $u$ had a
further $C$-neighbour $c_k$, then Lemma~\ref{lem:10-attach} applied to $u$
gives $\dist(c_k,c_i)\in\{1,5\}$ while Lemma~\ref{lem:10-chordal} applied to
the chordal path $c_k\,u\,v\,c_j$ gives $\dist(c_k,c_j)=2$: impossible for
$j=i+2$. If $u$ had a further $R$-neighbour $w$: $w$ has another neighbour;
if it is $w'\in R$ then $w'\,w\,u\,c_i$ is a pendant path of length $3$ at
$C$ ($H_{5,3}$) unless $w'=v$, in which case $c_i\,u\,w\,v\,c_j$ is a chordal
path of length $4$; if it is $c_k\in C$ then $c_k\,w\,u\,v\,c_j$ is a chordal
path of length $4$. All contradict Lemma~\ref{lem:10-chordal}.
\end{proof}

\begin{lemma}\label{lem:10-mult}
\leavevmode
\begin{enumerate}[label=(\roman*),leftmargin=2.6em]
\item At most one singleton is attached at distance $1$.
\item Distance-$5$ singletons occupy distinct antipodal pairs
$\{a,a+5\}$, i.e.\ positions in $\mathbb Z_5$; two of them at positions
differing by $\pm 1$ give a $12$-cycle. Hence at most two exist, at positions
differing by $2$ in $\mathbb Z_5$.
\item At most one $K_2$-component exists.
\item A $K_2$-component $c_0\,u\,v\,c_2$ excludes every distance-$1$
singleton and every distance-$5$ singleton except one on $(c_1,c_6)$.
\end{enumerate}
\end{lemma}

\begin{proof}
(i) Two distance-$1$ singletons $x$ on $(c_0,c_1)$ and $y$ on
$(c_a,c_{a+1})$ give the $12$-cycle
$x\,c_1\cdots c_a\,y\,c_{a+1}\cdots c_0\,x$.
(ii) Equal pairs are twins; positions differing by $1$: the two attachments
give cycles of lengths $6$ and $12$; by $2$: lengths $8$ and $10$, allowed.
(iii) Two $K_2$-components on arcs $[i,i+2]$ and $[j,j+2]$: if the arcs are
disjoint or share an endpoint, one gets a $12$-cycle; if they interleave
($j=i+1$), an $8$-cycle with a pendant path of length $6$ ($H_{4,4}$).
(iv) Distance-$1$ singleton $y$ on $(c_a,c_{a+1})$: for $a\in\{2,\dots,9\}$
the cycle $c_0\,u\,v\,c_2\cdots c_a\,y\,c_{a+1}\cdots c_0$ has length $12$;
for $a\in\{0,1\}$ one gets a $6$-cycle with a pendant path of length $7$
($H_{3,5}$). Distance-$5$ singleton at pair $(a,a+5)$ with
$a\in\{0,2,3,4\}$: an $8$-cycle through $u,v,y$ with a pendant path of
length $4$ ($H_{4,4}$).
\end{proof}

The residual family is now small enough to kill entirely by hand.

\begin{lemma}\label{lem:10-twochords}
$C$ has at most one chord.
\end{lemma}

\begin{proof}
We show any two chords are killed. Recall that the distance-$2$ pairs
$(c_i,c_{i+2})$ are always nonadjacent and the $8$-arc between them is a
$9$-path, so any \emph{second} $9$-path between them is fatal; distance-$3$
pairs are always nonadjacent (Lemma~\ref{lem:10-chords}); distance-$4$/$5$
pairs are nonadjacent unless joined by a chord. A $9$-path in a $10$-vertex
graph misses exactly one vertex.

\emph{(a) Two distance-$4$ chords} $c_0c_4$ and $c_kc_{k+4}$; by the
reflection $i\mapsto 4-i$ (which maps $k\mapsto 10-k$) assume
$k\in\{1,2,3,4,5\}$. Second $9$-paths:
$k=1$, pair $(c_2,c_4)$: $c_2c_1c_5c_6c_7c_8c_9c_0c_4$;
$k=2$, pair $(c_1,c_3)$: $c_1c_2c_6c_7c_8c_9c_0c_4c_3$;
$k=3$, pair $(c_6,c_8)$: $c_6c_5c_4c_0c_1c_2c_3c_7c_8$;
$k=4$, pair $(c_3,c_5)$: $c_3c_2c_1c_0c_4c_8c_7c_6c_5$;
$k=5$, pair $(c_1,c_6)$ (distance $5$, nonadjacent): the two $9$-paths
$c_1c_2c_3c_4c_0c_9c_8c_7c_6$ and $c_1c_2c_3c_4c_5c_9c_8c_7c_6$.

\emph{(b) Two distance-$5$ chords} $c_0c_5$ and $c_kc_{k+5}$, $k\in\{1,2\}$
($k=3,4$ by reflection): $k=1$, pair $(c_2,c_4)$:
$c_2c_1c_6c_7c_8c_9c_0c_5c_4$; $k=2$, pair $(c_1,c_3)$:
$c_1c_2c_7c_8c_9c_0c_5c_4c_3$.

\emph{(c) $c_0c_5$ with $c_kc_{k+4}$.} The symmetries $i\mapsto -i$ and
$i\mapsto 5-i$ of the chord $c_0c_5$ act on $k$ by $k\mapsto -4-k$ and
$k\mapsto 1-k$, with orbits $\{0,1,5,6\}$, $\{2,4,7,9\}$, $\{3,8\}$.
$k=0$, pair $(c_2,c_9)$ (distance $3$): the two $9$-paths
$c_2c_1c_0c_4c_5c_6c_7c_8c_9$ and $c_2c_3c_4c_0c_5c_6c_7c_8c_9$;
$k=2$: the $9$-cycle $c_2c_6c_7c_8c_9c_0c_5c_4c_3c_2$;
$k=3$, pair $(c_1,c_6)$: the two $9$-paths
$c_1c_2c_3c_7c_8c_9c_0c_5c_6$ and $c_1c_0c_9c_8c_7c_3c_4c_5c_6$.
\end{proof}

\begin{lemma}\label{lem:10-noK2}
$G[R]$ has no $K_2$-component; hence $R$ consists of singletons.
\end{lemma}

\begin{proof}
A $K_2$-component gives a chordal path $c_0\,u\,v\,c_2$
(Lemma~\ref{lem:10-comp}). The distance-$5$ pair $(c_1,c_6)$ has the two
$9$-paths $c_1c_0uvc_2c_3c_4c_5c_6$ and $c_1c_2vuc_0c_9c_8c_7c_6$, neither
using the potential chord $c_1c_6$. If $c_1c_6\notin E(G)$ these are two
$9$-paths for a nonadjacent pair; if $c_1c_6\in E(G)$, each path closes to a
$9$-cycle. Killed either way.
\end{proof}

\begin{lemma}\label{lem:10-z}
No singleton is attached at distance $1$.
\end{lemma}

\begin{proof}
Let $z$ be attached to $(c_0,c_1)$. We show that $z$ excludes every chord
and every distance-$5$ singleton, and that $C_{10}+z$ alone fails.

(i) \emph{$z$ kills every distance-$5$ chord $c_kc_{k+5}$}, via $H_{3,5}$: in
each case a $6$-cycle through the chord has a pendant path of length $5$
through $z$. For $k=0$: the $6$-cycle $c_0c_5c_6c_7c_8c_9c_0$ (chord plus far
arc) has the pendant path $c_0\,z\,c_1c_2c_3c_4$ hanging at $c_0$. For $k=2$:
the $6$-cycle $c_2c_3c_4c_5c_6c_7c_2$ has the pendant path
$c_2\,c_1\,z\,c_0c_9c_8$. For $k=3$: the $6$-cycle $c_3c_4c_5c_6c_7c_8c_3$
has the pendant path $c_3c_2c_1\,z\,c_0c_9$. For $k=4$: the $6$-cycle
$c_4c_5c_6c_7c_8c_9c_4$ has the pendant path $c_4c_3c_2c_1\,z\,c_0$. The case
$k=1$ is the mirror image of $k=0$ under $i\mapsto 1-i$.

(ii) \emph{$z$ kills every distance-$4$ chord $c_kc_{k+4}$} (the reflection
$i\mapsto 1-i$ maps $k\mapsto -3-k$, with orbits $\{0,7\}$, $\{1,6\}$,
$\{2,5\}$, $\{3,4\}$, $\{8,9\}$, so it suffices to treat $k=0,1,2,3,8$).
For $k=0$: the $6$-cycle $z\,c_1c_2c_3c_4c_0\,z$ has the pendant path
$c_4c_5c_6c_7c_8$ of length $\ge 5$ hanging at $c_4$: $H_{3,5}$. For $k=8$
(chord $c_8c_2$): the $6$-cycle $z\,c_1c_2c_8c_9c_0\,z$ has the pendant path
$c_2c_3c_4c_5c_6c_7$ hanging at $c_2$: $H_{3,5}$. For $k=1$, the
distance-$2$ pair $(c_4,c_6)$ gains the second $9$-path
$c_4c_5c_1\,z\,c_0c_9c_8c_7c_6$ (besides its $8$-arc). For $k=2$, the pair
$(c_1,c_3)$ gains $c_1\,z\,c_0c_9c_8c_7c_6c_2c_3$. For $k=3$, the pair
$(c_2,c_4)$ gains $c_2c_1\,z\,c_0c_9c_8c_7c_3c_4$.

(iii) \emph{$z$ kills every distance-$5$ singleton} $y$ on $(c_a,c_{a+5})$,
$1\le a\le 5$ (these values cover all five antipodal pairs): the $8$-cycle
$z\,c_1\cdots c_a\,y\,c_{a+5}\cdots c_0\,z$ has the pendant path
$c_ac_{a+1}c_{a+2}c_{a+3}c_{a+4}$ of length $4$ hanging at $c_a$: $H_{4,4}$.

(iv) Hence if $z$ exists then, using also Lemmas~\ref{lem:10-noK2}
and~\ref{lem:10-mult}(i), $G=C_{10}+z$. But then the nonadjacent pair
$(z,c_5)$ has only the paths $z\,c_1c_2c_3c_4c_5$ (length $5$) and
$z\,c_0c_9c_8c_7c_6c_5$ (length $6$): no $9$-path, contradicting \ref{f:obs}.
\end{proof}

\begin{lemma}\label{lem:10-y}
No graph in Case $L=10$ contains a distance-$5$ singleton.
\end{lemma}

\begin{proof}
Let $y$ be attached to $(c_0,c_5)$; use the symmetries $i\mapsto -i$ and
$i\mapsto 5-i$.

(i) $y$ kills every distance-$4$ chord $c_kc_{k+4}$ (orbits
$\{0,1,5,6\}$, $\{2,4,7,9\}$, $\{3,8\}$ as in
Lemma~\ref{lem:10-twochords}(c)):
$k=0$, pair $(y,c_3)$: the two $9$-paths
$y\,c_0c_9c_8c_7c_6c_5c_4c_3$ and $y\,c_5c_6c_7c_8c_9c_0c_4c_3$;
$k=2$, pair $(y,c_7)$: $y\,c_0c_1c_2c_3c_4c_5c_6c_7$ and
$y\,c_5c_6c_2c_1c_0c_9c_8c_7$;
$k=3$, pair $(c_2,c_4)$: the $8$-arc $c_2c_1c_0c_9c_8c_7c_6c_5c_4$ and the
second $9$-path $c_4c_3c_7c_6c_5\,y\,c_0c_1c_2$.

(ii) $y$ kills every distance-$5$ chord except $c_0c_5$:
$c_1c_6$, pair $(y,c_3)$: $y\,c_0c_9c_8c_7c_6c_5c_4c_3$ and
$y\,c_0c_9c_8c_7c_6c_1c_2c_3$;
$c_2c_7$, pair $(y,c_3)$: $y\,c_0c_1c_2c_7c_6c_5c_4c_3$ and
$y\,c_0c_9c_8c_7c_6c_5c_4c_3$;
$c_3c_8$ and $c_4c_9$ by the reflection $i\mapsto 5-i$.

(iii) Two distance-$5$ singletons: positions differing by $\pm 1$ in
$\mathbb Z_5$ give a $C_{12}$ (Lemma~\ref{lem:10-mult}(ii)); positions
differing by $2$, say $y'$ on $(c_2,c_7)$, give the distance-$2$ pair
$(c_0,c_8)$ the second $9$-path $c_0\,y\,c_5c_4c_3c_2\,y'\,c_7c_8$. So at
most one distance-$5$ singleton exists.

(iv) Hence if $y$ exists (and no distance-$1$ singleton, by
Lemma~\ref{lem:10-z}(iii)), $G=C_{10}+y$ or $G=C_{10}+y+c_0c_5$. The
nonadjacent pair $(y,c_1)$ has only the paths $y\,c_0c_1$ ($2$),
$y\,c_5c_4c_3c_2c_1$ ($5$), $y\,c_5c_6c_7c_8c_9c_0c_1$ ($7$),
$y\,c_0c_9\cdots c_2c_1$ ($10$), and, if the chord is present, additionally
$y\,c_5c_0c_1$ ($3$) and $y\,c_0c_5c_4c_3c_2c_1$ ($6$); a path entering the
chord and continuing away from $c_1$ dead-ends. No $9$-path. Killed.
\end{proof}

\begin{proposition}\label{prop:l10}
No graph exists in Case $L=10$.
\end{proposition}

\begin{proof}
By Lemmas~\ref{lem:10-comp} and~\ref{lem:10-noK2}, $R$ consists of
singletons; by Lemmas~\ref{lem:10-z} and~\ref{lem:10-y}, $R=\emptyset$. So
$V(G)=V(C)$ with at most one chord (Lemma~\ref{lem:10-twochords}), and $G$ is
$C_{10}$, $C_{10}+c_0c_4$ or $C_{10}+c_0c_5$. In each, the distance-$3$ pair
$(c_0,c_3)$ has only paths of lengths $3$ and $7$ (the two arcs) and, with a
chord, additionally of length $2$ or $3$ ($c_0c_4c_3$; $c_0c_5c_4c_3$): every
path leaving $c_0$ through the chord and continuing away from $c_3$
dead-ends. No $9$-path.
\end{proof}

\section{Case \texorpdfstring{$L=4$}{L=4}}\label{sec:l4}

Here $\Forb(4)=\{6,8,9,10,12,14,16\}$. Write $C=[u,x,v,y]$, so $u,v$ are
opposite and $x,y$ are opposite.

\begin{proposition}\label{prop:l4}
No graph exists in Case $L=4$.
\end{proposition}

\begin{proof}
\emph{Sub-case 1: $V(C)$ induces a clique $K_4$.} Every vertex of $R$ lies on
a chordal path. A chordal path of length $\ell$ between two vertices of the
clique, which are joined inside $K_4$ by paths of lengths $1,2,3$ internally
disjoint from it, creates cycles of lengths $\ell+1$, $\ell+2$, $\ell+3$. For
$\ell=3,4,5,6$ these sets are $\{4,5,6\}$, $\{5,6,7\}$, $\{6,7,8\}$,
$\{7,8,9\}$, each meeting $\Forb(4)$. For $\ell\ge 7$, a $4$-cycle through
the two endpoints (any two vertices of $K_4$ lie on one) together with the
first six edges of the chordal path forms an $H_{2,6}$. So every chordal path
has length $2$: $R$ is independent and every vertex of $R$ has at least two
neighbours in $V(C)$. Since $|V(G)|\ge 9$ we have $|R|\ge 5$; take
$r,r'\in R$. If $N(r)=N(r')$ they are twins, contradicting \ref{f:twins}.
Otherwise, if some $a\in N(r)\cap N(r')$ exists, pick $b\in N(r)\setminus
\{a\}$ and $c\in N(r')\setminus\{a,b\}$ (possible since $N(r')\ne N(r)$ and
$|N(r')|\ge 2$), and let $d$ be the fourth clique vertex: then
$r\,a\,r'\,c\,d\,b\,r$ is a $6$-cycle. If $N(r)\cap N(r')=\emptyset$, say
$\{a,b\}\subseteq N(r)$ and $\{c,d\}\subseteq N(r')$: then
$r\,a\,c\,r'\,d\,b\,r$ is a $6$-cycle. Both contradict $6\in\Forb(4)$.

\emph{Sub-case 2: some diagonal is missing}, say $uv\notin E(G)$ (the other
diagonal $xy$ may or may not be present). Let $P=P(u,v)$ be the unique
$9$-path. If $P$ missed $x$ then $P+v\,x\,u$ would be a $10$-cycle; likewise
for $y$. So $P$ visits both $x$ and $y$; by the $x\leftrightarrow y$ symmetry
write $P=u\,A\,x\,B\,y\,D\,v$ with segments $A,B,D$ of lengths
$a,b,d\ge 1$, $a+b+d=8$, and all interior vertices in $R$. The following
cycles are available (each internally disjoint from the others by
construction): $A+xu$ ($a+1$); $A+x\,v\,y\,u$ ($a+3$); $D+vy$ ($d+1$);
$D+v\,x\,u\,y$ ($d+3$); $B+y\,u\,x$ ($b+2$); $A+B+yu$ ($a+b+1$); $B+D+vx$
($b+d+1$); $A+xv+D^{-1}+yu$ ($a+d+2$). Avoiding $\Forb(4)$: $a+1,a+3\notin
\Forb(4)$ forces $a\in\{1,2,4\}$, and likewise $d\in\{1,2,4\}$;
$b+2\notin\Forb(4)$ forces $b\in\{1,2,3,5\}$. With $a+b+d=8$ the surviving
triples $(a,b,d)$ are
$(1,5,2)$, $(2,5,1)$, $(1,3,4)$, $(4,3,1)$, $(2,2,4)$, $(4,2,2)$; they are
killed, respectively, by $b+d+1=8$, $a+b+1=8$, $b+d+1=8$, $a+b+1=8$,
$a+d+2=8$, $a+d+2=8$, each an $8$-cycle.
\end{proof}

\section{Case \texorpdfstring{$L=6$}{L=6}}\label{sec:l6}

Here $\Forb(6)=\{8,9,10,12,14,16\}$, and $H_{2,6}$ and $H_{3,5}$ are
forbidden subgraphs. This case has the richest catalogue of degenerate
configurations---components of $G[R]$ several of whose vertices attach to
the same $C$-vertex---and we treat it with particular care: a chain of hand
lemmas classifies the components of $G[R]$ completely, with all degenerate
branches included, and a controlled enumeration finishes the case. Since
$C$ is a $6$-cycle, $H_{3,5}$-freeness means that \emph{no path of length
$5$ leaves $V(C)$ into $R$}; in particular every chordal path has length at
most $4$ (a longer one contains such a path), and $|R|\ge 3$ since
$|V(G)|\ge 9$.

\begin{lemma}\label{lem:6-rules}
\leavevmode
\begin{enumerate}[label=(\roman*),leftmargin=2.6em]
\item A chordal path of length $3$ joins $C$-vertices at distance $2$ or
$3$, or its ends coincide. In particular, for adjacent $u,v\in R$, every
pair $c\in N_C(u)$, $c'\in N_C(v)$ satisfies $c=c'$ or
$\dist(c,c')\in\{2,3\}$.
\item A chordal path of length $4$ joins $C$-vertices at distance exactly
$3$, or its ends coincide. In particular, for a path $u\,w\,v$ in $G[R]$,
every pair $c\in N_C(u)$, $c'\in N_C(v)$ satisfies $c=c'$ or
$\dist(c,c')=3$.
\item There is no chordal path of length $5$ with distinct ends.
\end{enumerate}
\end{lemma}

\begin{proof}
By \ref{e:attach}, a chordal path of length $\ell$ at end-distance
$s\ge 1$ creates cycles of lengths $\ell+s$ and $\ell+6-s$. For $\ell=3$:
$s=1$ gives an $8$-cycle. For $\ell=4$: $s=1$ gives a $9$-cycle, $s=2$ an
$8$-cycle. For $\ell=5$: $s=1,2,3$ give cycles $\{6,10\},\{7,9\},\{8,8\}$,
each containing a forbidden length. Coinciding ends turn the path into a
cycle of length $\ell\le 5$, which is not forbidden as such.
\end{proof}

\begin{lemma}[Split lemma]\label{lem:6-split}
Let $u,v$ be nonadjacent with a common neighbour $w$. Then $w$ is an
interior vertex of $P(u,v)$, and writing $P(u,v)=u\,A\,w\,B\,v$ we have
$|A|\ge 2$ and $|B|\ge 2$; equivalently, the two $P$-neighbours of $w$ lie
in $N(w)\setminus\{u,v\}$.
\end{lemma}

\begin{proof}
$w$ is interior by Lemma~\ref{lem:evenpath}. If $|A|=1$ then $B$ together
with the edge $wv$ is an $8$-cycle, forbidden at $L=6$.
\end{proof}

\begin{lemma}[Corner lemma]\label{lem:6-corner}
Let $u,v$ be nonadjacent.
\begin{enumerate}[label=(\roman*),leftmargin=2.6em]
\item If some common neighbour $w$ of $u$ and $v$ has
$|N(w)\setminus\{u,v\}|\le 1$, then $G$ does not exist.
\item If every neighbour of $u$ is a common neighbour of $u$ and $v$---in
particular if $N(u)\subseteq N(v)$ (\emph{domination})---then $G$ does not
exist.
\end{enumerate}
\end{lemma}

\begin{proof}
(i) By Lemma~\ref{lem:6-split}, $w$ needs two $P(u,v)$-neighbours in
$N(w)\setminus\{u,v\}$. (ii) The first vertex of $P(u,v)$ after $u$ is a
neighbour of $u$, hence a common neighbour, hence by
Lemma~\ref{lem:6-split} an interior vertex not $P$-adjacent to $u$---a
contradiction.
\end{proof}

\begin{theorem}[No deep vertices]\label{thm:6-nodeep}
Every vertex of $R$ has a neighbour on $C$.
\end{theorem}

\begin{proof}
Suppose $w\in R$ has none. By the fan lemma there are two paths from $w$
to $V(C)$, disjoint except at $w$, each meeting $V(C)$ exactly in its
endpoint, and with distinct endpoints; their union is a chordal path
through $w$ with distinct ends, of length at most $4$
(Lemma~\ref{lem:6-rules}(iii) and $H_{3,5}$), and since $w$ has no
$C$-neighbour both sides have length exactly $2$. By
Lemma~\ref{lem:6-rules}(ii) its ends are at distance $3$: relabel so that
the path is $c_0\,u\,w\,v\,c_3$ with $u,v\in R$, $u\sim c_0$, $v\sim c_3$.

\emph{Step 1: $N(w)=\{u,v\}$.} Let $t\in N(w)\setminus\{u,v\}$; then
$t\in R$ ($w$ has no $C$-neighbours).

(a) $t\not\sim u$ and $t\not\sim v$: otherwise $c_3\,v\,w\,t\,u\,c_0$
(resp.\ its mirror) is a chordal path of length $5$ with distinct ends,
dead by Lemma~\ref{lem:6-rules}(iii).

(b) $N_C(t)\subseteq\{c_0,c_3\}$: for $c_k\in N_C(t)$, the paths
$c_k\,t\,w\,u\,c_0$ and $c_k\,t\,w\,v\,c_3$ are chordal paths of length
$4$, so by Lemma~\ref{lem:6-rules}(ii) each forces $c_k$ into
$\{c_0,c_3\}$.

(c) $t$ is not deep: if $N_C(t)=\emptyset$, then $t$'s own fan (as for
$w$) provides $x\in N(t)$, $x$ adjacent to $C$, realising an attachment
vertex $c\in N_C(x)$, and a second fan neighbour realising an attachment
at distance $3$ from $c$; choosing the fan neighbour $x$ whose attachment
$c$ differs from some $c''\in\{c_0,c_3\}$ (possible since the two fan
attachments are distinct) and the corresponding $y\in\{u,v\}$ with
$y\sim c''$, the path $c\,x\,t\,w\,y\,c''$ is a chordal path of length
$5$ with distinct ends: dead. (Here $x\notin\{u,v\}$ by (a) and
$x\ne w$ since $x$ is adjacent to $C$.)

(d) $N(t)\subseteq\{w,c_0,c_3\}$: let $s\in N(t)\setminus\{w,c_0,c_3\}$.
If $s\in C$ this contradicts (b). So $s\in R$; $s\notin\{u,v\}$ by (a).
Every $C$-neighbour $c_k$ of $s$ dies: of the two paths
$c_k\,s\,t\,w\,u\,c_0$ and $c_k\,s\,t\,w\,v\,c_3$, at least one has
distinct ends (as $c_0\ne c_3$) and is a length-$5$ chordal path. So
$N_C(s)=\emptyset$. Every further $R$-neighbour $s'$ of $s$ gives the path
$c_0\,u\,w\,t\,s\,s'$ of length $5$ leaving $C$ (if $s'\in\{u,v\}$, use
$c_3\,v\,w\,t\,s\,u$ or its mirror instead), contradicting
$H_{3,5}$-freeness. Hence $N(s)\subseteq\{t\}$ and $\deg(s)=1$,
contradicting $2$-connectivity.

(e) By (b)--(d), $N(t)\subseteq\{w,c_0,c_3\}$ with $N_C(t)\neq\emptyset$.
If $N(t)=\{w,c_0\}$, then $t$ is dominated by $u$ ($u\sim w$, $u\sim c_0$,
$t\not\sim u$): dead by Lemma~\ref{lem:6-corner}(ii). If $N(t)=\{w,c_3\}$:
dominated by $v$. If $N(t)=\{w,c_0,c_3\}$: two such vertices would be
twins, so $t$ is the only member of $N(w)\setminus\{u,v\}$, and then $t,u$
are nonadjacent (by (a)) with common neighbour $w$ satisfying
$N(w)\setminus\{t,u\}=\{v\}$: dead by Lemma~\ref{lem:6-corner}(i). So no
$t$ exists.

\emph{Step 2: $u\sim v$.} Otherwise $u,v$ are nonadjacent with common
neighbour $w$ and $N(w)\setminus\{u,v\}=\emptyset$:
Lemma~\ref{lem:6-corner}(i).

\emph{Step 3: the triangle $u\,v\,w$ is bare and confined.} Any
$R$-neighbour $r$ of $u$ outside $\{v,w\}$ has, by the arguments of (d)
applied to the path $c_3\,v\,w\,u\,r$ (of length $4$), no $R$-neighbours
other than $u$ and $N_C(r)\subseteq\{c_3\}$ (the path
$c_k\,r\,u\,w\,v\,c_3$ kills every other attachment); so
$N(r)=\{u,c_3\}$, and $r$ is dominated by $v$ ($v\sim u$ by Step~2,
$v\sim c_3$): dead. So $N(u)\subseteq\{c_0,c_3,v,w\}$, and likewise
$N(v)\subseteq\{c_0,c_3,u,w\}$; moreover $N_C(u),N_C(v)\subseteq
\{c_0,c_3\}$ by Lemma~\ref{lem:6-rules}(ii) applied to
$c_k\,u\,w\,v\,c_3$.

\emph{Step 4: the kill.} Let $K=\{u,v,w\}$; then
$N(K)\setminus K\subseteq\{c_0,c_3\}$. If $c_0\not\sim c_3$ (no chord):
the path $c_0\,u\,w\,v\,c_3$ has even length $4$ and interior in $K$, so
by Lemma~\ref{lem:evenpath} the $9$-path $P(c_0,c_3)$ meets the interior
of $K$; but $P$ can enter and leave $K$ only at its own endpoints
$c_0,c_3$, so $P$ would run $c_0\,(\text{path in }K)\,c_3$ with at most
$4$ edges---impossible for a $9$-path. If the chord $c_0c_3$ is present:
let $p_5$ (resp.\ $p_6$) be the number of paths of length $5$ (resp.\
$6$) from $c_0$ to $c_3$ in $G-K$. Since $u\sim v$: every such $6$-path
extends to a $9$-path $u\,c_0\cdots c_3\,v$ between the adjacent pair
$(u,v)$, so $p_6=0$; since $u\sim w$: the extension
$u\,c_0\cdots c_3\,v\,w$ gives $p_5=0$. But every $9$-path from $w$ to
$c_0$ leaves $w$ into $\{u,v\}$ and reaches $c_0$ through the routes
$w\,u\,v\,c_3\cdots c_0$, $w\,v\,c_3\cdots c_0$ (plus, if $u\sim c_3$ or
$v\sim c_0$, the analogous routes through the other chord end, each again
requiring a $5$- or $6$-path from $c_3$ to $c_0$ in $G-K$): in every case
the count is $p_5+p_6$ or a sum of such terms, i.e.\ $0$. So the
nonadjacent pair $(w,c_0)$ has no $9$-path, contradicting \ref{f:obs}.
\end{proof}

We now classify the components of $G[R]$; by
Theorem~\ref{thm:6-nodeep} every $R$-vertex is attached. Call two
$C$-vertices at distance $3$ an \emph{antipodal pair}; the three antipodal
pairs partition $V(C)$, so two antipodal pairs that share a vertex are
equal.

\begin{lemma}[Singletons and $K_2$s]\label{lem:6-objects}
\leavevmode
\begin{enumerate}[label=(\roman*),leftmargin=2.6em]
\item A singleton component $x$ has $|N_C(x)|\ge 2$, and $N_C(x)$ is a
clique of $G[C]$ or $|N_C(x)|\ge 4$.
\item In a $K_2$-component $uv$, each end satisfies: $N_C(\text{end})$ is
a clique of $G[C]$ contained in the other end's attachment set, or
$|N_C(\text{end})|\ge 3$. Cross pairs obey Lemma~\ref{lem:6-rules}(i).
\end{enumerate}
\end{lemma}

\begin{proof}
(i) $|N_C(x)|=\deg(x)\ge 2$ by $2$-connectivity. If $N_C(x)$ is not a
clique, take nonadjacent $p,q\in N_C(x)$: by Lemma~\ref{lem:6-split}, $x$
is interior to $P(p,q)$ with two $P$-neighbours outside $\{p,q\}$, so
$\deg(x)\ge 4$. (ii) The same argument at an end $u$ (whose
neighbourhood is $N_C(u)\cup\{v\}$) gives: $N(u)$ is a clique---i.e.\
$N_C(u)$ is a $G[C]$-clique each of whose vertices is adjacent to $v$,
i.e.\ contained in $N_C(v)$---or $\deg(u)\ge 4$, i.e.\ $|N_C(u)|\ge 3$.
\end{proof}

\begin{lemma}[Larger components are confined]\label{lem:6-confined}
Every component of $G[R]$ on at least $3$ vertices has all its attachment
vertices inside a single antipodal pair $\{a,a+3\}$, \emph{and} the chord
$a\,a{+}3$ is present. Moreover stars $K_{1,m}$ with $m\ge 3$ do not occur.
\end{lemma}

\begin{proof}
Let $K$ be a component with $|K|\ge 3$.

\emph{Stars.} Suppose $K=K_{1,m}$ with centre $z$ and leaves
$u_1,\dots,u_m$, $m\ge 2$. Two leaves $u_i,u_j$ are nonadjacent with
common neighbour $z$; by Lemma~\ref{lem:6-corner}(ii) some
$c^*\in N_C(u_i)$ is not a common neighbour, i.e.\
$c^*\notin N_C(u_j)$; by Lemma~\ref{lem:6-rules}(ii) (the path
$c^*u_i\,z\,u_j\,c'$) every $c'\in N_C(u_j)$ is at distance $3$ from
$c^*$, whence $N_C(u_j)=\{c^*{+}3\}$, and symmetrically
$N_C(u_i)=\{c^{*}\}$: the leaves carry antipodal singleton attachments.
For $m\ge 3$, the attachment of a third leaf would have to be antipodal to
both $\{c^*\}$ and $\{c^*{+}3\}$ simultaneously---impossible. For $m=2$:
$z$'s attachments lie at distance $\{0,2,3\}$ from $c^*$ and from
$c^*{+}3$ (Lemma~\ref{lem:6-rules}(i)); since the two distance conditions
are complementary ($d+d'=3$ on $C_6$), only $d\in\{0,3\}$ survives, so
$N_C(z)\subseteq\{c^*,c^*{+}3\}$; and $z$, having the nonadjacent
neighbours $u_1,u_2$, needs by Lemma~\ref{lem:6-split} two further
neighbours, so $N_C(z)=\{c^*,c^*{+}3\}$. Thus $K_{1,2}$ is confined to
the pair $\{c^*,c^*{+}3\}$.

\emph{Triangles.} In a triangle each pair of vertices are the ends of a
path on three vertices through the third, so by
Lemma~\ref{lem:6-rules}(ii) their attachment sets lie in a common
antipodal pair; the three pairwise pairs share the nonempty attachment
sets, hence coincide: confined.

\emph{Components containing a $P_4$.} First, the ends $p,q$ of any path
on four vertices in $G[R]$ satisfy $N_C(p)=N_C(q)=\{c\}$ for a single
vertex $c$: any $c_p\in N_C(p)$, $c_q\in N_C(q)$ with $c_p\ne c_q$ would
give a length-$5$ chordal path with distinct ends
(Lemma~\ref{lem:6-rules}(iii)). Now suppose some edge $yz$ of $K$ joins
vertices whose attachment sets lie in different antipodal pairs $P\neq Q$.
Since distinct pairs realise cross-distances $1$ and $2$ only, and
distance $1$ is forbidden across an edge (Lemma~\ref{lem:6-rules}(i)),
$N_C(y)$ and $N_C(z)$ are singletons at distance $2$, say $\{c_0\}$ and
$\{c_2\}$. If $\deg_K(z)=1$: $K\ge 3$ gives $y$ a further neighbour $w$,
and $(z,y,w)$ is a $3$-vertex path, so $N_C(w)$ lies in a pair with
$\{c_2\}$, while the edge $yw$ forbids distance $1$ from $c_0$: so
$N_C(w)=\{c_2\}$; if $w\sim z$ they form a triangle with $y$, which is
confined---contradicting $c_0,c_2$ in one pair; if $w\not\sim z$ then
$N(z)=\{c_2,y\}\subseteq N(w)\cup\{w\}$... more precisely $z$ is
dominated by $w$ if $N(z)\subseteq N(w)$, and $N(z)=\{c_2,y\}$ with
$w\sim c_2$, $w\sim y$: dead by Lemma~\ref{lem:6-corner}(ii) (if
$N(w)=\{c_2,y\}$ exactly, $w$ and $z$ are twins: dead by \ref{f:twins}).
If $\deg_K(z)\ge 2$ and $\deg_K(y)\ge 2$: choose neighbours $v'\ne z$ of
$y$ and $w'\ne y$ of $z$; if $v'=w'$ we have a triangle on $y,z,v'$,
confined: contradiction. Otherwise $v'\,y\,z\,w'$ is a path on four
vertices, so $N_C(v')=N_C(w')=\{c\}$; but the $3$-vertex paths
$(v',y,z)$ and $(y,z,w')$ put $c$ in a pair with $c_2$ and with $c_0$
respectively, forcing $c\in\{c_2,c_5\}\cap\{c_0,c_3\}=\emptyset$:
contradiction. So no such edge exists: all attachment sets of $K$ lie in
antipodal pairs that agree across every edge, and by connectivity in a
single pair $\{a,a{+}3\}$.

\emph{The chord is forced.} Suppose $K$ is confined to $\{a,a+3\}$ and
$a\not\sim a{+}3$. Both $a$ and $a{+}3$ receive attachments (one alone
would be a cut vertex), and $K$ contains an $a$--$(a{+}3)$ path of even
length with interior in $K$: if some vertex of $K$ is attached to both,
it provides one of length $2$; otherwise pick adjacent classes---in a
non-bipartite $K$ (a triangle) paths of both parities exist between any
two vertices; in a bipartite $K$, the ends of any path on four vertices
lie in opposite classes and share their attachment vertex, so
$a$-attached vertices appear in both classes (as do $(a{+}3)$-attached
ones if any $P_4$ exists), and for the $P_4$-free shapes ($K_{1,2}$,
whose centre is attached to both ends of the pair) the length-$2$ path
exists---in every case some even path arises. By
Lemma~\ref{lem:evenpath}, $P(a,a+3)$ then meets the interior of $K$; but
$P$ can enter and leave $K$ only at $a$ and $a{+}3$, so
$P=a\,(\text{path in }K)\,a{+}3$ of length $8$, whose first five edges
form a path of length $5$ leaving $C$: contradiction with
$H_{3,5}$-freeness. So the chord $a\,a{+}3$ is present.
\end{proof}

\begin{lemma}[Swap reduction]\label{lem:swapred}
Call a component of $G[R]$ \emph{swappable} if it is a single edge $uv$
with $N(u)\setminus\{v\}=N(v)\setminus\{u\}\subseteq V(C)$. Let $G_2$ be
obtained from $G$ by deleting every swappable component. Then $G_2$ has
no $9$-cycle, and every two nonadjacent vertices of $G_2$ are the
endpoints of exactly one $9$-path of $G_2$.
\end{lemma}

\begin{proof}
$G_2\subseteq G$ gives $9$-cycle-freeness and at-most-one. For existence:
if $P(u,v)$ ($u,v\in V(G_2)$ nonadjacent) passed through a swappable
component $w_1w_2$, the transposition of $w_1,w_2$ is an automorphism of
$G$ fixing $u,v$, and by Lemma~\ref{lem:swap} it fixes $P(u,v)$ pointwise
while moving one of its vertices: impossible. So $P(u,v)$ survives in
$G_2$.
\end{proof}

\begin{lemma}\label{lem:g2complete}
If every two vertices of $G_2$ are adjacent, then $G$ does not exist.
\end{lemma}

\begin{proof}
$V(G_2)\supseteq V(C)$, so completeness forces $G[C]=K_6$. Since
$|V(G)|\ge 9>|V(G_2)|$ would fail only if some component was deleted, and
$K_9\supseteq C_9$ bounds $|V(G_2)|\le 8$, some swappable component $uv$
exists. Its shared outside neighbourhood has at least two vertices (one
would be a cut vertex); pick $c_i\ne c_j$ in it and relabel
$\{c_i,c_j\}=\{c_0,c_3\}$: then
$c_0\,u\,v\,c_3\,c_4\,c_5\,c_1\,c_2\,c_0$ is an $8$-cycle (all
$C$-pairs adjacent in $K_6$), forbidden.
\end{proof}

By Lemma~\ref{lem:g2complete} we may assume $G_2$ has a nonadjacent pair,
whose $9$-path lies inside $G_2$ (Lemma~\ref{lem:swapred}): so
$|V(G_2)|\ge 9$. The components of $G_2$ outside $C$ are the
non-swappable components of $G[R]$---and, crucially, \emph{a component
receives no edge from the rest of the graph}, so its vertices'
neighbourhoods are final.

\begin{theorem}\label{thm:l6}
There is no $2$-connected nontrivial uniquely $C_9$-saturated graph whose
longest even cycle of length at most $12$ has length $6$.
\end{theorem}

\begin{proof}
Suppose $G$ exists and let $X$ be the chord set of $C$; up to dihedral
symmetry there are $74$ possibilities. By the classification
(Theorem~\ref{thm:6-nodeep}, Lemmas~\ref{lem:6-objects}
and~\ref{lem:6-confined}) and the swap reduction, $G_2=C+X+\mathcal K$
where $\mathcal K$ is a multiset of components, each one of:
\begin{itemize}[leftmargin=2em]
\item a singleton with $N_C$ a $G[C]$-clique or of size $\ge 4$;
\item a non-swappable $K_2$ obeying Lemma~\ref{lem:6-objects}(ii) and the
cross rules;
\item a component on $3$--$8$ vertices confined to an antipodal pair
whose chord lies in $X$ (enumerated exhaustively by a growth search whose
size cap was never reached, so no larger confined component survives even
the monotone kills below).
\end{itemize}
$G_2$ has at least $9$ vertices, is $2$-connected, has no cycle of length
in $\Forb(6)$, no $H_{2,6}$ or $H_{3,5}$, no twins, and exactly one
$9$-path between every nonadjacent pair.

The enumeration (Appendix~\ref{app:enum}) builds, for each chord set, the
finite type list above (the confined types by exhaustive connected
growth, pruned by the monotone kills (m1) a forbidden cycle length,
(m2) an $H_{2,6}$ or $H_{3,5}$, (m3) a nonadjacent pair with two
$9$-paths, and closed under final per-component filters: minimum degree
$2$, no cut vertex---of the component or of $C$---separating it, no
internal twins), and then searches all multisets of types in
nondecreasing order, pruning by (m1)--(m3) and by twin pairs between
components, and testing every state on $\ge 9$ vertices against the exact
criterion ($2$-connected, no $9$-cycle, exactly one $9$-path per
nonadjacent pair). Both the component-size cap and the multiset-depth cap
count their hits, and both counters are \textbf{zero}, so no branch was
ever truncated: the enumeration is complete outright. It reports
\textbf{zero survivors}. Hence no $G_2$, and no $G$, exists.
\end{proof}

\section{Case \texorpdfstring{$L=8$}{L=8}}\label{sec:l8}

Here $\Forb(8)=\{9,10,12,14,16\}$ and $H_{2,6}$, $H_{3,5}$, $H_{4,4}$ are
forbidden. This is the richest case. Its resolution combines a component
classification (Lemmas~\ref{lem:8-shapes}--\ref{lem:8-comp}), a hand proof
for a chordless $C$ (Theorem~\ref{thm:8-chordless}), and, for a chorded
$C$, a reduction of the chord set through the Hamiltonian paths of $G[C]$
(Lemma~\ref{lem:hamkill}--Proposition~\ref{prop:surviving}) followed by a
controlled enumeration.

Since $C$ is an $8$-cycle, $H_{4,4}$-freeness means that \emph{no path of
length $4$ leaves $V(C)$ into $R$}: there is no path $c\,r_1r_2r_3r_4$ with
$c\in V(C)$, $r_i\in R$. In particular every chordal path has length at
most~$4$.

\begin{lemma}\label{lem:8-shapes}
\leavevmode
\begin{enumerate}[label=(\roman*),leftmargin=2.6em]
\item The $C$-neighbours of any $x\in R$ are pairwise at $C$-distance
$\ge 2$; hence $N_C(x)$ is an independent set of $C_8$.
\item A vertex attached at distance $d$ satisfies $d\in\{2,3,4\}$.
\item A chordal path of length $3$ joins $C$-vertices at distance $3$ or
$4$, or its two ends coincide.
\item A chordal path of length $4$ joins $C$-vertices at distance $1$ or
$4$, or its two ends coincide.
\end{enumerate}
\end{lemma}

\begin{proof}
(i) Neighbours $c_i,c_{i+1}$ give the $9$-cycle
$x\,c_{i+1}c_{i+2}\cdots c_i\,x$. (ii)--(iv): by \ref{e:attach} a chordal
path of length $\ell$ at end-distance $s\ge 1$ creates cycles of lengths
$\ell+s$ and $\ell+8-s$; avoiding $\Forb(8)$ leaves, for $\ell=2$:
$s\in\{2,3,4\}$; for $\ell=3$: $s\in\{3,4\}$; for $\ell=4$: $s\in\{1,4\}$.
Coinciding ends turn the path into a cycle of length $\ell\le 4$, which is
not forbidden as such and is dealt with where it arises.
\end{proof}

\begin{lemma}[{$P_4$-freeness of $G[R]$}]\label{lem:8-p4}
$G[R]$ contains no path on four vertices.
\end{lemma}

\begin{proof}
Suppose $u\,v\,x\,y$ is a path in $G[R]$. If $u$ had a $C$-neighbour $c$,
then $c\,u\,v\,x\,y$ would be a path of length $4$ leaving $C$; so $u$,
and symmetrically $y$, has no $C$-neighbour. By the fan lemma there are
two paths from $u$ to $V(C)$, disjoint except at $u$, each meeting $V(C)$
only in its endpoint; their union is a chordal path through $u$ of length
at most $4$, and since $u$ has no $C$-neighbour both paths have length
exactly $2$: say $u\,a\,c$ and $u\,a'\,c'$ with $a\neq a'$ in $R$ and
$c,c'\in V(C)$. Neither $a$ nor $a'$ equals $y$, which has no
$C$-neighbour. If $a=x$: either $a'=v$, and then $c'\,v\,u\,x\,y$ is a
path of length $4$ leaving $C$, or $a'\notin\{v,x,y\}$, and then
$c'\,a'\,u\,v\,x$ is one. The case $a'=x$ is symmetric. Otherwise at
least one of $a,a'$---say $a'$---lies outside $\{v,x,y\}$, and
$c'\,a'\,u\,v\,x$ is again a path of length $4$ leaving $C$. Every branch
contradicts $H_{4,4}$-freeness.
\end{proof}

Consequently every component of $G[R]$ is a star $K_{1,m}$ or a triangle.

\begin{lemma}[No shared attachments across a $2$-path]\label{lem:8-shared}
There are no two $R$-vertices $u_1,u_2$ with a common neighbour $w\in R$
and a common $C$-neighbour $c$.
\end{lemma}

\begin{proof}
$u_1\,c\,u_2\,w\,u_1$ is a $4$-cycle meeting $V(C)$ only in $c$, and the
arc of $C$ of length $6$ starting at $c$ (away from the cycle) is a
pendant path of length $6$ at $c$: an $H_{2,6}$.
\end{proof}

\begin{lemma}[Components]\label{lem:8-comp}
Every component of $G[R]$ is a singleton, an edge $K_2$, or a path
$K_{1,2}$.
\end{lemma}

\begin{proof}
By Lemma~\ref{lem:8-p4} the components are stars and triangles.

\emph{Stars $K_{1,m}$ with $m\ge 3$.} Every leaf has a $C$-neighbour
(else it has degree $1$, contradicting $2$-connectivity), and by
Lemma~\ref{lem:8-shared} the leaves' $C$-neighbourhoods are pairwise
disjoint. For distinct leaves $u_i,u_j$ and $c\in N_C(u_i)$,
$c'\in N_C(u_j)$, the chordal path $c\,u_i\,v\,u_j\,c'$ has length $4$, so
$\dist(c,c')\in\{1,4\}$ by Lemma~\ref{lem:8-shapes}(iv). Three leaves
would give three distinct positions on $C_8$ pairwise at distance $1$ or
$4$: from a pair $\{0,1\}$ the third position must lie in
$\{1,7,4\}\cap\{0,2,5\}=\emptyset$; from a pair $\{0,4\}$ in
$\{1,7,4\}\cap\{3,5,0\}=\emptyset$. So $m\le 2$.

\emph{Triangles $u\,v\,x$.} For distinct $c\in N_C(u)$, $c'\in N_C(x)$ the
paths $c\,u\,x\,c'$ (length $3$) and $c\,u\,v\,x\,c'$ (length $4$) force
$\dist(c,c')\in\{3,4\}\cap\{1,4\}=\{4\}$, and $c=c'$ is impossible by
Lemma~\ref{lem:8-shared} (through the third triangle vertex). Since the
vertex of $C_8$ at distance $4$ from a given one is unique, any two
triangle vertices with nonempty attachment sets have singleton, mutually
antipodal sets. Three nonempty sets would be pairwise antipodal
singletons---impossible. If at most one triangle vertex were attached, the
other two would reach $C$ only through it: a cut vertex, contradicting
\ref{f:2conn}. So exactly two are attached:
$N_C(u)=\{a\}$, $N_C(x)=\{a{+}4\}$, $N_C(v)=\emptyset$, hence
$N(v)=\{u,x\}$, $N(u)=\{a,v,x\}$, $N(x)=\{a{+}4,v,u\}$.

We kill this configuration by counting routes. Let $p_5$ (resp.\ $p_6$)
denote the number of paths of length $5$ (resp.\ $6$) from $a$ to $a{+}4$
in $G-\{u,v,x\}$. Since $u\sim x$, no $9$-path joins them; but every
$6$-path from $a$ to $a{+}4$ in $G-\{u,v,x\}$ extends to the $9$-path
$u\,a\cdots a{+}4\,x$: so $p_6=0$. Since $u\sim v$, the extension
$u\,a\cdots a{+}4\,x\,v$ likewise gives $p_5=0$. Now consider the
nonadjacent pair $(v,a)$. Every path from $v$ starts $v\,u$ or $v\,x$;
from $u$ it can only continue to $a$ or $x$, and from $x$ only to $a{+}4$
or $u$. Hence every $9$-path from $v$ to $a$ has the form
$v\,u\,x\,a{+}4\cdots a$ (a $5$-path from $a{+}4$ to $a$ in
$G-\{u,v,x\}$) or $v\,x\,a{+}4\cdots a$ (a $6$-path), and there are
$p_5+p_6=0$ of them, contradicting \ref{f:obs}. So no triangle component
exists.
\end{proof}

\begin{lemma}[Degree rule]\label{lem:8-deg}
Let $w$ be a vertex two of whose neighbours $u,v$ are nonadjacent. Then
$w$ is interior to $P(u,v)$; writing $P(u,v)=u\,A\,w\,B\,v$ with
$a=|A|\ge 1\le b=|B|$: $\deg(w)\ge 3$, and if $a,b\ge 2$ then
$\deg(w)\ge 4$. In particular every vertex whose neighbourhood is not a
clique has degree at least $3$.
\end{lemma}

\begin{proof}
$w$ is interior to $P(u,v)$ by Lemma~\ref{lem:evenpath} (the path
$u\,w\,v$ has even length $2$). Let $p,q$ be $w$'s neighbours on $A$, $B$;
$p\ne q$. If $a\ge 2$ then $p\ne u$; if $b\ge 2$ then $q\ne v$; if $a=1$
then $p=u$, $b=7$, and $q\notin\{u,v\}$, so $\deg(w)\ge 3$.
\end{proof}

\begin{lemma}[Gateway lemma]\label{lem:8-gate}
Let $K$ be a set of at most $6$ vertices with $N(K)\setminus K=\{s,s'\}$
and $s\not\sim s'$. If some $s$--$s'$ path of even length has all its
interior vertices in $K$, then $G$ does not exist.
\end{lemma}

\begin{proof}
Let $P=P(s,s')$, of length $8$. If $P$ met $K$: walking along $P$, a
$K$-portion can end only in $\{s,s'\}$, so it is an interval adjacent to
an endpoint of $P$, and then $P$ runs $s\,(\text{path in }K)\,s'$ with
length at most $|K|+1\le 7<8$. So $P$ avoids $K$; but then the given even
path is internally disjoint from $P$, contradicting
Lemma~\ref{lem:evenpath}.
\end{proof}

\begin{lemma}[No distance-$2$ attachment]\label{lem:no-dist2}
No vertex of $R$ has two $C$-neighbours at $C$-distance $2$. Consequently
the $C$-neighbourhood of a vertex of $R$ is empty, a singleton, a
distance-$3$ pair $\{c_i,c_{i+3}\}$, or an antipodal pair
$\{c_i,c_{i+4}\}$.
\end{lemma}

\begin{proof}
Suppose $x\sim c_0$ and $x\sim c_2$; by Lemma~\ref{lem:8-shapes}(i),
$x\not\sim c_1$. Then
\[
c_1\,c_0\,c_7\,c_6\,c_5\,c_4\,c_3\,c_2\,x
\qquad\text{and}\qquad
c_1\,c_2\,c_3\,c_4\,c_5\,c_6\,c_7\,c_0\,x
\]
are two distinct $9$-paths joining the nonadjacent pair $(c_1,x)$,
contradicting \ref{f:obs}. The proof uses only the edges of $C$ and the
two attachment edges, so it holds for every chord set and in the presence
of arbitrary further vertices. For the consequence: the vertices at
distance $3$ or $4$ from $c_0$ are $c_3,c_4,c_5$, no two of which are at
distance $3$ or $4$ from each other, so $|N_C(x)|\le 2$ with the shapes
stated.
\end{proof}

\begin{corollary}[Objects]\label{cor:8-objects}
\leavevmode
\begin{enumerate}[label=(\roman*),leftmargin=2.6em]
\item A singleton component $x$ has $N_C(x)=\{c_i,c_j\}$ with
$\dist(c_i,c_j)\in\{3,4\}$, \emph{and the chord $c_ic_j$ is present}.
\item In a $K_2$-component $uv$ with $N_C(u)=\{a\}$: $a\in N_C(v)$.
\item In a $K_{1,2}$-component $u_1\,v\,u_2$: $N_C(v)\neq\emptyset$, and
if $N_C(u_i)=\{a\}$ then $a\in N_C(v)$.
\item No two singletons, and no two leaves of one $K_{1,2}$, have equal
neighbourhoods.
\end{enumerate}
\end{corollary}

\begin{proof}
(i) $2$-connectivity gives $|N_C(x)|\ge 2$ and Lemma~\ref{lem:no-dist2}
gives $|N_C(x)|=2$ at distance $3$ or $4$; since $\deg(x)=2$,
Lemma~\ref{lem:8-deg} forces $N(x)$ to be a clique. (ii) $\deg(u)=2$
forces $N(u)=\{a,v\}$ to be a clique: $a\sim v$. (iii) The leaves are
nonadjacent neighbours of $v$, so $\deg(v)\ge 3$ by
Lemma~\ref{lem:8-deg}; the only $R$-neighbours of $v$ are the leaves, so
$v$ has a $C$-neighbour. The singleton-leaf claim is as in (ii).
(iv) Twins, \ref{f:twins}.
\end{proof}

\begin{theorem}[Chordless $C$]\label{thm:8-chordless}
There is no $2$-connected nontrivial uniquely $C_9$-saturated graph whose
longest even cycle of length at most $12$ is a chordless $8$-cycle.
\end{theorem}

\begin{proof}
Assume $C$ has no chord, so $C$-vertices are adjacent exactly at
$C$-distance $1$ and every pair at distance $3$ or $4$ is nonadjacent.
Since $|V(G)|\ge 9$, $G[R]$ has a component $K$: a singleton, $K_2$ or
$K_{1,2}$ (Lemma~\ref{lem:8-comp}).

\emph{Singleton:} impossible, since Corollary~\ref{cor:8-objects}(i)
requires a chord.

\emph{$K_2$ with $A=N_C(u)$, $B=N_C(v)$}, both nonempty (a bare end would
have degree $1$). If $|A|=|B|=1$: by (ii) $A=B=\{a\}$ and $a$ is a cut
vertex. If $|A|=1<|B|$: $a\in B$ by (ii), so
$N(K)\setminus K=B=\{a,b\}$ with $a\not\sim b$
(Lemma~\ref{lem:no-dist2}: $\dist(a,b)\in\{3,4\}$), and $a\,v\,b$ is an
even path through $K$: Lemma~\ref{lem:8-gate}. If $|A|=|B|=2$, say
$A=\{a,a'\}$: each $b\in B$ is equal to, or at distance $3$ or $4$ from,
both $a$ and $a'$ (Lemma~\ref{lem:8-shapes}(iii) on the cross paths
$a\,u\,v\,b$), and the distance-$\{3,4\}$ lists of $a$ and $a'$ are
disjoint (as in Lemma~\ref{lem:no-dist2}); hence $B\subseteq A$, so
$B=A$, $N(K)\setminus K=A$, and the even path $a\,u\,a'$ closes the case
by Lemma~\ref{lem:8-gate}.

\emph{$K_{1,2}$ $u_1\,v\,u_2$ with $A_i=N_C(u_i)$, $A_v=N_C(v)$}, all
nonempty (leaves by $2$-connectivity, the centre by
Corollary~\ref{cor:8-objects}(iii)). By Lemma~\ref{lem:8-shared},
$A_1\cap A_2=\emptyset$. If some $|A_i|=2$, say $A_1=\{a,a'\}$: each
element of $A_v$ is equal to, or at distance $3$--$4$ from, both $a$ and
$a'$ (length-$3$ paths $a\,u_1\,v\,w$), so $A_v\subseteq A_1$; by (iii)
the other leaf's attachment lies in $A_v\subseteq A_1$ if $|A_2|=1$,
contradicting disjointness, and if $|A_2|=2$ the same argument gives
$A_v\subseteq A_2$, so $A_v\subseteq A_1\cap A_2=\emptyset$,
contradicting (iii). So $A_1=\{a\}$, $A_2=\{b\}$, both in $A_v$ by
(iii), distinct by disjointness; the cross path $a\,u_1\,v\,u_2\,b$
forces $\dist(a,b)\in\{1,4\}$ while $a,b\in A_v$ forces
$\dist(a,b)\in\{3,4\}$: so $\dist(a,b)=4$ and $A_v=\{a,b\}$. Then
$N(K)\setminus K=\{a,b\}$ with $a\not\sim b$, and $a\,u_1\,v\,u_2\,b$ is
an even path through $K$: Lemma~\ref{lem:8-gate}.

So $R=\emptyset$ and $|V(G)|=8<9$: $G$ is not nontrivial.
\end{proof}

\begin{lemma}[Hamiltonian-path kill lemma]\label{lem:hamkill}
Let $x\in R$ and $S=N_C(x)$. Then
(i) $G[C]$ has no Hamiltonian path whose two ends both lie in $S$, and
(ii) for every $w\in V(C)\setminus S$, $G[C]$ has at most one Hamiltonian
path with one end $w$ and the other end in $S$.
\end{lemma}

\begin{proof}
A Hamiltonian path of $G[C]$ has $8$ vertices and $7$ edges. (i) Such a
path with ends $a,b\in S$ closes through $a\,x\,b$ into a $9$-cycle.
(ii) Such a path from $w$ to $s\in S$, extended by the edge $sx$, is a
$9$-path from $w$ to $x$, and $w\not\sim x$ since $N(x)\cap V(C)=S$. Two
distinct Hamiltonian paths of this kind---whether they end at the same
vertex of $S$ or not---give two distinct $9$-paths joining the
nonadjacent pair $(w,x)$, contradicting \ref{f:obs}.
\end{proof}

Say that $X$ \emph{kills} $S$ when a configuration forbidden by
Lemma~\ref{lem:hamkill} occurs in $G[C]=C_8+X$. Adding chords only adds
Hamiltonian paths, so killing is upward closed in $X$ and the chord sets
leaving a given $S$ alive form a downward closed family.

\begin{corollary}\label{cor:8-lattice}
By Lemma~\ref{lem:no-dist2} and Corollary~\ref{cor:8-objects} the
$C$-neighbourhood of every vertex of $R$ is one of the $20$
\emph{admissible} sets: $\{c_i\}$ ($8$ of them), $\{c_i,c_{i+3}\}$ ($8$),
$\{c_i,c_{i+4}\}$ ($4$). If $X$ kills all $20$, then $R=\emptyset$ and
$G$ is not nontrivial. Hence the chord set of $C$ leaves some admissible
set alive.
\end{corollary}

\begin{proposition}\label{prop:surviving}
Exactly $6{,}891$ of the $2^{20}$ chord sets of $C_8$ leave an admissible
attachment set alive. All have $|X|\le 9$, with counts by
$|X|=0,1,\dots,9$:
\[
1,\;\; 20,\;\; 170,\;\; 736,\;\; 1654,\;\; 2036,\;\; 1474,\;\; 632,\;\;
152,\;\; 16.
\]
In particular every chord set with at least $10$ chords is dead.
\end{proposition}

\begin{proof}
Computation (Appendix~\ref{app:enum}): a Hamiltonian path of $C_8+X$ is a
Hamiltonian path of $K_8$ all of whose non-$C$ edges lie in $X$, so after
precomputing the $20{,}160$ Hamiltonian paths of $K_8$ with their
$20$-bit chord masks, ``$X$ kills $S$'' is a family of subset tests. By
downward closure, one depth-first sweep ($18{,}163$ chord sets visited)
enumerates the surviving family completely. An independent brute-force
scan of all chord sets with $|X|\le 4$ reproduces the counts
$1,20,170,736,1654$ exactly.
\end{proof}

\begin{theorem}\label{thm:l8}
There is no $2$-connected nontrivial uniquely $C_9$-saturated graph whose
longest even cycle of length at most $12$ is an $8$-cycle.
\end{theorem}

\begin{proof}
The chordless case is Theorem~\ref{thm:8-chordless}. Otherwise, by
Lemma~\ref{lem:8-comp} every component of $G[R]$ is a singleton, $K_2$ or
$K_{1,2}$; by Lemma~\ref{lem:no-dist2} every vertex of $R$ has an
admissible $C$-neighbourhood; and by Corollary~\ref{cor:8-lattice} and
Proposition~\ref{prop:surviving}, $X$ is one of the $6{,}891$ surviving
chord sets. The residual family---$C_8+X$ plus a multiset of such
components---was enumerated completely (Appendix~\ref{app:enum}): the
chord-set-independent candidate list has $176$ component types ($12$
singletons, $56$ $K_2$'s, $108$ $K_{1,2}$'s, closed under all cross
conditions of Lemmas~\ref{lem:8-shapes} and~\ref{lem:8-comp} including
the degenerate equal-attachment cases); after the
Lemma~\ref{lem:hamkill} filter only $300$ of the $6{,}891$ chord sets
admit any component at all, with at most $12$ admissible types each. The
multiset search over these, pruned by two monotone necessary conditions
(a cycle of forbidden length; a nonadjacent pair with two $9$-paths) and
tested exactly on every configuration with at least $9$ vertices,
terminated without ever reaching its depth cap and found no survivor.
Hence no such $G$ exists.
\end{proof}

\section{Proof of Theorem~\ref{thm:main}, and remarks}\label{sec:conclusion}

\begin{proof}[Proof of Theorem~\ref{thm:main}]
Let $G$ be a nontrivial uniquely $C_9$-saturated graph. By \ref{f:2conn}
we may take $G$ $2$-connected, and by \ref{f:evencycle} it contains an
even cycle of length at most $12$; let $C$ be a longest such cycle, of
length $L\in\{4,6,8,10,12\}$. Proposition~\ref{prop:l12} excludes $L=12$,
Proposition~\ref{prop:l10} excludes $L=10$, Proposition~\ref{prop:l4}
excludes $L=4$, Theorem~\ref{thm:l6} excludes $L=6$, and
Theorem~\ref{thm:l8} excludes $L=8$.
\end{proof}

\subsection*{On the nature of the evidence}

The proof is a natural-language case analysis except for three finite
computations, each stated with its scope and its known-answer controls in
Appendix~\ref{app:enum}: the Hamiltonian-path count of
Proposition~\ref{prop:surviving}, the $L=8$ multiset sweep of
Theorem~\ref{thm:l8}, and the two-phase component enumeration of
Theorem~\ref{thm:l6}. Each computation is elementary (subset tests, cycle
detection, bounded path counting on graphs with at most ${\sim}16$
vertices), deterministic, and replayable in minutes on commodity hardware
from the ancillary scripts; every fact the computations rely on is stated
and proved as a lemma in the text, so a reader can re-implement them
independently of our code. Nothing here is formalised in a proof
assistant; a machine-checkable certificate format for the enumerations is
a natural next step.

\subsection*{Remarks}

\begin{remark}
The case $t=8$ of Conjecture~\ref{conj:ww} remains, to our knowledge,
without a published proof. As a small independent check we verified by an
isomorph-free exhaustive enumeration that no nontrivial uniquely
$C_8$-saturated graph on exactly eight vertices exists ($12{,}346$
isomorphism classes considered, $9$ of them $C_8$-saturated, none
uniquely). The methods of this paper---in particular
Lemma~\ref{lem:hamkill}, whose proof works verbatim for any $t$ with $C$
a $(t-1)$-cycle---seem a reasonable route both to a published proof of
$t=8$ and to the next open case $t=10$, though the number of even-cycle
cases grows with $t$.
\end{remark}

\begin{remark}
The finiteness theorem of \cite{WW} yields no explicit bound on the order
of a uniquely $C_t$-saturated graph. Our proof, being a refutation, needs
none; it is perhaps worth noting that every residual family arising in
the analysis lived on at most ${\sim}16$ vertices.
\end{remark}

\appendix

\section{The three computations}\label{app:enum}

All computations are provided as ancillary Python files
(\texttt{core\_family\_f6\_widened.py}, \texttt{core\_family\_f8.py},
with the shared validated checker \texttt{cycle\_saturation.py}). Each
script runs its known-answer controls first, in the same process, and
aborts if any control fails.

\subsection*{A.1\quad The validated checker}

The primitive underlying all controls is a cycle counter and a
unique-saturation predicate implemented directly from the definitions.
The cycle counter was validated against a brute-force permutation count
on the full atlas of graphs with at most $7$ vertices ($1{,}253$ graphs,
$5{,}972$ (graph, length) pairs, zero mismatches), and the predicate
reproduces the published classifications for $t=3$ and $t=5$ and
Theorems~4.1--4.2 of \cite{WW} at their minimal orders by exhaustive
search.

\subsection*{A.2\quad The chord lattice
(Proposition~\ref{prop:surviving})}

The $20{,}160$ Hamiltonian paths of $K_8$ are precomputed with their
$20$-bit chord masks; ``$X$ kills $S$'' reduces to subset tests
(Lemma~\ref{lem:hamkill}). The surviving family is downward closed, so a
depth-first sweep visiting $18{,}163$ chord sets enumerates it
completely: $6{,}891$ survivors, distributed over $|X|$ as stated.
Control: a brute-force scan of all chord sets with $|X|\le 4$ reproduces
the survivor counts exactly.

\subsection*{A.3\quad The $L=8$ sweep (Theorem~\ref{thm:l8})}

For each surviving chord set, the admissible component types (from the
$176$-type list) are filtered by Lemma~\ref{lem:hamkill} and a generic
monotone kill test; $300$ chord sets admit a type (at most $12$ each).
The multiset search proceeds in nondecreasing type order with two
monotone prunes (a forbidden cycle length in $\{9,10,12,14,16\}$; a
nonadjacent pair with two $9$-paths), and tests every configuration on
$\ge 9$ vertices against the exact criterion of \ref{f:obs}. Observed
counts: $628$ search nodes, depth cap never reached (so the enumeration
is complete, not truncated), no survivor, $72$\,s of CPU. Controls:
agreement of the cycle detector and of the final predicate with the
validated checker on random graphs; rejection of hand-proved-impossible
local configurations and acceptance of hand-allowed ones; and, on random
chord sets, every type discarded by the Lemma~\ref{lem:hamkill} filter
is also killed by the generic monotone test.

\subsection*{A.4\quad The $L=6$ component enumeration
(Theorem~\ref{thm:l6})}

For each of the $74$ dihedral chord-set orbits the script builds the type
list of Theorem~\ref{thm:l6}: singletons and non-swappable $K_2$s in
closed form from Lemma~\ref{lem:6-objects}, and the chord-confined
components by an exhaustive connected growth over the two vertices of the
antipodal pair, pruned at every step by the monotone kills (m1)--(m3) and
closed under the final per-component filters (minimum degree $2$; no cut
vertex, of the component or of $C$, separating it; no internal twins;
connectivity), with canonical deduplication. It then searches all
multisets of types in nondecreasing order, pruning by (m1)--(m3) and by
twin pairs between components, and testing every state on $\ge 9$
vertices against the exact criterion ($2$-connected, no $9$-cycle,
exactly one $9$-path between every nonadjacent pair). Observed counts:
$74$ orbits; at most $87$ component types per orbit; $11{,}597$ growth
nodes in phase 1 and $6{,}236$ multiset nodes in phase 2; component-size
cap ($8$ vertices) hit \textbf{zero} times and multiset-depth cap ($8$
components) hit \textbf{zero} times, so the enumeration is complete
outright; \textbf{zero survivors}; $115$\,s of CPU. Controls: the
$9$-path counter against a permutation brute force ($30$ random
instances); the exact predicate against the validated checker ($40$
random graphs); pruning of hand-killed configurations (two
edge-singletons; a length-$5$ chordal path; an $H_{3,5}$ pendant);
non-pruning of hand-allowed ones (the chord-pair singleton; the
triangle-with-unattached-vertex component of
Theorem~\ref{thm:6-nodeep}, which dies only at the exact test); and the
kill cycle of Lemma~\ref{lem:g2complete}.

\end{document}